\documentclass[10pt,11pt]{amsart}%
\usepackage{graphicx}
\usepackage{amsmath}
\usepackage{color}
\usepackage{amsfonts}
\usepackage{amssymb}%
\providecommand{\U}[1]{\protect\rule{.1in}{.1in}}
\providecommand{\U}[1]{\protect\rule{.1in}{.1in}}
\providecommand{\U}[1]{\protect\rule{.1in}{.1in}}
\providecommand{\U}[1]{\protect\rule{.1in}{.1in}}
\providecommand{\U}[1]{\protect\rule{.1in}{.1in}}
\providecommand{\U}[1]{\protect\rule{.1in}{.1in}}
\providecommand{\U}[1]{\protect\rule{.1in}{.1in}}
\newtheorem{theorem}{Theorem}[section]

\newtheorem{corollary}[theorem]{Corollary}

\newtheorem{definition}[theorem]{Definition}
\newtheorem{example}[theorem]{Example}

\newtheorem{lemma}[theorem]{Lemma}

\newtheorem{proposition}[theorem]{Proposition}
\newtheorem{remark}[theorem]{Remark}
\newtheorem{remarks}[theorem]{Remarks}

\def \A{{\cal{A}}}

\begin{document}
\title[Alcove walks and $k$-Schur functions]{Alcove random walks, $k$-Schur functions and the minimal boundary of the
$k$-bounded partition poset}
\date{June, 2018}
\author{C\'{e}dric Lecouvey and Pierre Tarrago}

\begin{abstract}
We use $k$-Schur functions to get the minimal boundary of the $k$-bounded
partition poset. This permits to describe the central random walks on affine
Grassmannian elements of type $A$ and yields a rational expression for their
drift. We also recover Rietsch's parametrization of totally nonnegative
unitriangular Toeplitz matrices without using quantum cohomology of flag
varieties. All the homeomorphisms we define can moreover be made explicit by
using the combinatorics of $k$-Schur functions and elementary computations
based on Perron-Frobenius theorem.

\end{abstract}
\maketitle

\section{Introduction}

A function on the Young graph is harmonic when its value on any Young diagram
$\lambda$ is equal to the sum of its values on the Young diagrams obtained by
adding one box to $\lambda$. The set of extremal nonnegative such functions
(i.e. those that cannot be written as a convex combination) is called the
minimal boundary of the Young graph. It is homeomorphic to the Thoma simplex.
Kerov and Vershik proved that the extremal nonnegative harmonic functions give
the asymptotic characters of the symmetric group.\ O'Connell's results
\cite{OC1} also show that they control the law of some conditioned random
walks.\ In another but equivalent direction, Kerov-Vershik approach of these
harmonic functions yields both a simple parametrization of the set of infinite
totally nonnegative unitriangular Toeplitz matrices (see \cite{Ker}) and a
characterization of the morphisms from the algebra $\Lambda$ of symmetric
functions to $\mathbb{R}$ which are nonnegative on the Schur functions. These
results were generalized in \cite{LLP2} and \cite{LT}.\ A crucial observation
here is the connection between the Pieri rule on Schur functions and the
structure of the Young graph (which is then said multiplicative in
Kerov-Vershik terminology).

There is an interesting
$k$-analogue $\mathcal{B}_{k}$ of the Young lattice of partitions whose
vertices are the $k$-bounded partitions (i.e. those with no parts greater than
$k$). Its oriented graph structure is isomorphic to the Hasse poset on the
affine Grassmannian permutations of type $A$ which are minimal length coset
representatives in $\widetilde{W}/W$, where $\widetilde{W}$ is the affine type
$A_{k}^{(1)}$ group and $W$ the symmetric group of type $A_{k}$. The graph
$\mathcal{B}_{k}$ is also multiplicative but we have then to replace the
ordinary Schur functions by the $k$-Schur functions (see \cite{LLMSSZ} and the
references therein) and the algebra $\Lambda$ by $\Lambda_{(k)}=\mathbb{R}%
[h_{1},\ldots,h_{k}]$.\ The $k$-Schur functions were introduced by Lascoux,
Lapointe and Morse \cite{LLM} as a basis of $\Lambda_{(k)}$. It was
established by Lam \cite{LamLR} that their corresponding constant structures
(called $k$-Littlewood-Richardson coefficients) are nonnegative.\ This was
done by interpreting $\Lambda_{(k)}$ in terms of the homology ring of the
affine Grassmannian which, by works of Lam and Shimozono, can be conveniently
identified with the quantum cohomology ring of partial flag varieties studied
by Rietsch \cite{Ri}. By merging these two geometric approaches one can theoretically
deduce that the set of morphisms from $\Lambda_{(k)}$ to $\mathbb{R}$,
nonnegative on the $k$-Schur functions, are also parametrized by
$\mathbb{R}_{\geq0}^{k}$.

In this paper, we shall use another approach to avoid sophisticated geometric
notions and make our construction as effective as possible. Our starting point
is the combinatorics of $k$-Schur functions. We prove that they permit to get an
explicit parametrization of the morphisms $\varphi$ nonnegative on the
$k$-Schur functions, or equivalently of all the minimal $t$-harmonic functions
with $t\geq0$ on $\mathcal{B}_{k}$. Both notions are related by the simple
equality $t=\varphi(s_{(1)})$. Each such morphism is in fact completely
determined by its values $\vec{r}=(r_{1},\ldots,r_{k})\in\mathbb{R}_{\geq
0}^{k}$ on the Schur functions indexed by the rectangle partitions
$R_{a}=(k-a+1)^{a}.$ We get a bi-continuous (homeomorphism) parametrization
which is moreover effective in the sense one can compute from $\vec{r}$ the
values of $\varphi$ on any $k$-Schur function from the Perron Frobenius vector
of a matrix $\Phi$ encoding the multiplication by $s_{(1)}$ in $\Lambda_{(k)}%
$. Also, applying the primitive element theorem in the field of fractions of $\Lambda_{(k)}$ permits to prove that for any fixed
$t\geq0$, each $\varphi(s_{\lambda}^{(k)})$ is a rational function on
$\mathbb{R}_{\geq0}^{k}$.\ So, the only place where geometry is needed in this paper is
in Lam's proof of the nonnegativity of the $k$-Schur coefficients. As far as
we are aware a complete combinatorial $k$-Littlewood-Richardson rule is not
yet available (see nevertheless \cite{MS}).

Random walks on reduced alcoves paths have been considered by Lam in
\cite{Lam2}.\ They are random walks on a particular tessellation of
$\mathbb{R}^{k}$ by alcoves supported by hyperplanes, where each hyperplane
can be crossed only once.\ The random walks considered in this paper are
central and thus differ from those of \cite{Lam2}. Two trajectories with the
same ends will have the same probability (see \S\ref{SubSec_Compar} for a comparison between the two models). We characterize all the possible
laws of these alcove random walks and also get a simple algebraic expression
of their drift as a rational function on $\mathbb{R}_{\geq0}^{k}$. Our results
are more precisely summarized in the following Theorem.

\begin{figure}
\begin{center}
\scalebox{0.8}{\includegraphics[
height=7.0424cm,
width=7.0973cm
]
{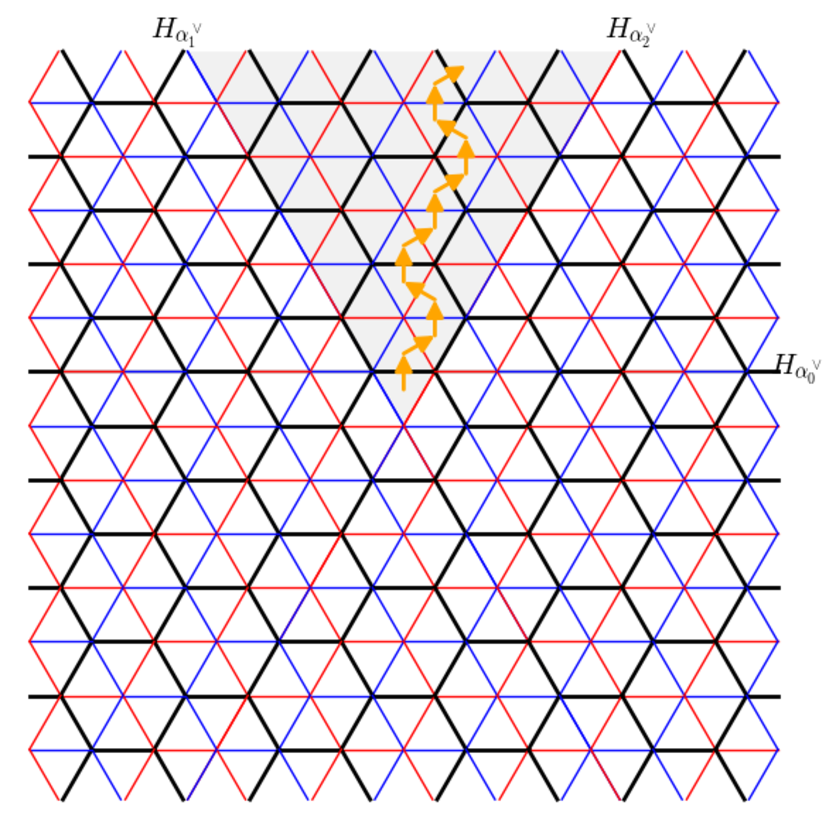}}
\\
\caption{A reduced alcove walk on Grassmannian elements for $k=2$}
\end{center}
\end{figure}

\begin{theorem}
\ 

\begin{enumerate}
\item To each $\vec{r}\in\mathbb{R}_{\geq0}^{k}$ corresponds a unique morphism
$\varphi:\Lambda_{(k)}\rightarrow\mathbb{R}$ nonnegative on the $k$-Schur
functions and such that $\varphi(s_{R_{a}})=r_{a}$ for any $a=1,\ldots,k.$

\item The previous one-to-one correspondence is a homeomorphism, and
$\varphi$ can be explicitly computed from $\vec{r}$ by using Perron
Frobenius theorem.

\item The minimal boundary of $\mathcal{B}_{k}$ is homeomorphic to a simplex
$\mathcal{S}_{k}$ of $\mathbb{R}_{\geq0}^{k}$.

\item To each $\vec{r}\in\mathcal{S}_{k}$ corresponds a central random walk
$(v_{n})_{n\geq0}$ on affine Grassmannian elements which verifies a law of
large numbers. The coordinates of its drift are the image by $\varphi$ of
rational fractions in the $k$-Schur functions. They are moreover rational on
$\mathcal{S}_{k}$.
\end{enumerate}
\end{theorem}

As in the case of the Young graph, the description of the minimal boundary of the graph $\mathcal{B}_{k}$ yields a parametrization of the set $T_{\geq 0}$ of infinite totally nonnegative unitriangular $(k+1)\times(k+1)$ Toeplitz matrices. In \cite{Ri}, Rietsch already obtained a parametrization for the variety $T_{\geq0}$ from the quantum cohomology of partial
flag varieties (see also \cite{Ri2} for an alternative construction using mirror symmetry). More precisely, such a matrix is proved to be completely
determined by the datum of its $k$ initial minors obtained by considering its
southwest corners. The main result of the present paper gives, as a corollary, an alternative proof of Rietsch's parametrization.
\begin{corollary}
To each $\vec{r}\in\mathbb{R}_{\geq0}^{k}$ corresponds a unique matrix
$M$ in $T_{\geq0}$ whose $k$ southwest initial minors are exactly
$r_{1},\ldots,r_{k}$. This correspondence yields a homeomorphism from $T_{\geq 0}$ to $\mathbb{R}_{\geq0}^{k}$, and each matrix $M$ in $T_{\geq 0}$ can be constructed from $\vec{r}$ by using Perron
Frobenius theorem.
\end{corollary}

The paper is organized as follows. In Section 2, we recall some background on
alcoves, partitions and $k$-Schur functions. In Section 3, we introduce the
matrix $\Phi$ and study its irreducibility. Section 4 uses classical tools of
field theory to derive an expression of any $k$-Schur function in terms of
$s_{(1)}$ and the $s_{R_{a}}$ $a=1,\ldots,k$. We get the parametrization of
all the minimal $t$-harmonic functions defined on $\mathcal{B}_{k}$ by
$\mathbb{R}_{\geq0}^{k}$ in Section 5.\ In Section 6, we give the law of
central random walks on alcoves and compute their drift by exploiting a
symmetry property of the matrix $\Phi$. Finally, Section 7 presents
consequences of our results, notably we rederive Rietsch's result on finite
Toeplitz matrices, establish rational expressions for the $\varphi(s_{\lambda
}^{(k)})$, characterize the simplex $\mathcal{S}_{k}$ and show the inverse
limit of the minimal boundaries of the graphs $\mathcal{B}_{k},k\geq2$ is the
Thomas simplex.

Although we restrict ourselves to type $A$ in this paper, we expect that  
our approach can be extended to other types notably by using the results of
\cite{LSS} and \cite{P} (see Section 8).

\section{Harmonic functions on the lattice of $k$-bounded partitions}

\subsection{The lattices $\mathcal{C}_{l}$ and $\mathcal{B}_{k}$.}

\label{subset_Lattices}In this section, we refer to \cite{LLMSSZ} and
\cite{Mac} for the material which is not defined.\ Fix $l>1$ a nonnegative
integer and set $k=l-1$. Let $\widetilde{W}$ be the affine Weyl group of type
$A_{k}^{(1)}$. As a Coxeter group, $\widetilde{W}$ is generated by the
reflections $s_{0},s_{1},\ldots,s_{k}$ so that its subgroup generated by
$s_{1},\ldots,s_{k}$ is isomorphic to the symmetric group $S_{l}$. Write
$\ell$ for the length function on $\widetilde{W}$. The group $\widetilde{W}$
determines a Coxeter arrangement by considering the hyperplanes orthogonal to
the roots of type $A_{k}^{(1)}$. The connected components of this hyperplane
arrangement yield a tessellation of $\mathbb{R}^{k}$ by alcoves on which the
action of $\widetilde{W}$ is regular.\ We denote by $A^{(0)}$ the fundamental
alcove. Write $\widetilde{R}$ for the set of affine roots of type $A_{k}%
^{(1)}$ and $R$ for its subset of classical roots of type $A_{k}$. The simple
roots are denoted by $\alpha_{0},\ldots,\alpha_{k}$ and $P$ is the weight
lattice of type $A_{k}$ with fundamental weights $\Lambda_{1},\ldots
,\Lambda_{k}$.

A reduced alcove path is a sequence of alcoves $(A_{1},\ldots,A_{m})$ such
that $A_{1}=A^{(0)}$ and for any $i=1,\ldots,m-1$, the alcoves $A_{i+1}$ and
$A_{i}$ share a common face contained in a hyperplane $H_{i}$ so that the
sequence $H_{1},\ldots,H_{m-1}$ is without repetition (each hyperplane can be
crossed only once).\ In the sequel, all the alcove paths we shall consider
will be reduced. For any $i=1,\ldots,m-1$, let $w_{i}$ be the unique element
of $\widetilde{W}$ such that $A_{i}=w_{i}(A^{(0)})$.\ Write $\vartriangleleft$
for the weak Bruhat order on $\widetilde{W}$ and $\rightarrow$ for the
covering relation $w\rightarrow w^{\prime}$ if and only if $w\vartriangleleft
w^{\prime}$ and $\ell(w^{\prime})=\ell(w)+1$. We then have $w_{1}\rightarrow
w_{2}\rightarrow\cdots\rightarrow w_{m}$.

We shall identify a partition and its Young diagram. Recall that a $l$-core
can be seen as a partition where no box has hook length equal to $l$. Given a
$l$-core $\lambda$, we denote by $\ell(\lambda)$ its length which is equal to
the number of boxes of $\lambda$ with hook length less that $l$. Recall that
the residue of a box in a Young diagram is the difference modulo $l$ between
its row and column indices. We can define an arrow $\lambda\overset
{i}{\rightarrow}\mu$ between the two $l$-cores $\lambda$ and $\mu$ when
$\lambda\subset\mu$ and all the boxes in $\mu/\lambda$ have the same residue
$i$. By forgetting the label arrows, we get the structure of a graded rooted
graph $\mathcal{C}_{l}$ on the $l$-cores. For any two vertices $\lambda
\rightarrow\mu$ we have $\ell(\mu)=\ell(\lambda)+1$. Nevertheless, the
difference between the rank of the partitions $\lambda$ and $\mu$ is not
immediate to get in general.

The affine Grassmannian elements are the elements $w\in\widetilde{W}$ whose
associated alcoves are exactly those located in the fundamental Weyl chamber
(that is, in the Weyl chamber containing the fundamental alcove $A^{(0)}%
$).\ The $l$-cores are known to parametrize the affine Grassmannian elements.
More precisely, given two $l$-cores such that $\lambda\overset{i}{\rightarrow
}\mu$ and $w$ the affine Grassmannian element associated to $\lambda$,
$w^{\prime}=ws_{i}$ is the affine Grassmannian element associated to $\mu
$.\ In particular, we get $\ell(\lambda)=\ell(w)$. So reduced alcove paths in
the fundamental Weyl chamber, saturated chains of affine Grassmannian elements
and paths in $\mathcal{C}_{l}$ naturally correspond.

\bigskip

A $k$-bounded partition is a partition $\lambda$ such that $\lambda_{1}\leq
k$. There is a simple bijection between the $l$-cores and the $k$-bounded
partitions. Start with a $l$ core $\lambda$ and delete all the boxes in the
diagram of $\lambda$ having a hook length greater than $l$ (recall there is no
box with hook length equal to $l$ since $\lambda$ is a $l$-core). This gives a
skew shape and to obtain a partition, move each row so obtained on the left.
The result is a $k$-bounded partition denoted $\mathfrak{p}(\lambda)$. For
some examples and the converse bijection $\mathfrak{c}$, see \cite{LLMSSZ}
pages 18 and 19. This bijection permits to define an analogue of conjugation
for the $k$-bounded partitions. Given a $k$-bounded partition $\kappa$ set%
\[
\kappa^{\omega_{k}}=\mathfrak{p}(\mathfrak{c}(\kappa)^{\prime}).
\]
The graph $\mathcal{B}_{k}$ is the image of the graph $\mathcal{C}_{l}$ under
the bijection $\mathfrak{p}$. This means that $\mathcal{B}_{k}$ is the graph
obtained from $\mathcal{C}_{l}$ by deleting all the boxes with hook length
greater that $l$ and next by aligning the rows obtained on the left. In
particular, reduced alcove paths in the fundamental Weyl chamber correspond to
$k$-bounded partitions paths in $\mathcal{B}_{k}$. We have the following lemma:

\begin{lemma}
We have an arrow $\kappa\rightarrow\delta$ in $\mathcal{B}_{k}$ if and only if
$\left\vert \delta\right\vert =\left\vert \kappa\right\vert +1,$
$\kappa\subset\delta$ and $\kappa^{\omega_{k}}\subset\delta^{\omega_{k}}%
$.\footnote{So $\mathcal{B}_{k}$ \emph{should not be confused} with the
subgraph of the Young graph with vertices the $k$-bounded partitions.}
\end{lemma}

Let $\Lambda$ be\ the algebra of symmetric functions in infinitely many
variables over $\mathbb{R}$. It is endowed with a scalar product $\langle
\cdot,\cdot\rangle$ such that $\langle s_{\lambda},s_{\mu}\rangle
=\delta_{\lambda,\mu}$ for any partitions $\lambda$ and $\mu$. Let
$\Lambda_{(k)}$ be the subalgebra of $\Lambda$ generated by the complete
homogeneous functions $h_{1},\ldots,h_{k}$. In particular, $\{h_{\lambda}%
\mid\lambda$ is $k$-bounded$\}$ is a basis of $\Lambda_{(k)}$.

\subsection{The $k$-Schur functions}

We now define a distinguished basis of $\Lambda_{(k)}$ related to the graph
structures of $\mathcal{C}_{l}$ and $\mathcal{B}_{k}$. Consider $\lambda$ and
$\mu$ two $k$-bounded partitions with $\lambda\subset\mu$ and $r\leq k$ a
positive integer.

\begin{definition}
\label{Def_Hori_Strip}We will say that $\mu/\lambda$ is a weak horizontal
strip of size $r$ when

\begin{enumerate}
\item $\mu/\lambda$ is an horizontal strip with $r$ boxes (i.e. the boxes in
$\mu/\lambda$ belong to different columns),

\item $\mu^{\omega_{k}}/\lambda^{\omega_{k}}$ is a vertical strip with $r$
boxes (i.e. the boxes in $\mu^{\omega_{k}}/\lambda^{\omega_{k}}$ belong to
different rows).
\end{enumerate}
\end{definition}

Let us now define the notion of $k$-bounded semistandard tableau of shape
$\lambda$ a $k$-bounded partition and weight $\alpha=(\alpha_{1},\ldots
,\alpha_{d})$ a composition of $\left\vert \lambda\right\vert $ \emph{with no
part larger than }$k$.\ 

\begin{definition}
A $k$-bounded semistandard tableau of shape $\lambda$ is a semistandard
filling of $\lambda$ with integers in $\{1,\ldots,d\}$ such that for any
$i=1,\ldots,d$ the boxes containing $i$ define a weak horizontal strip of size
$\alpha_{i}$.
\end{definition}

One can prove that for any $k$-bounded partitions $\lambda$ and $\alpha$ the
number $K_{\lambda,\alpha}^{(k)}$ of $k$-bounded semistandard tableaux of
shape $\lambda$ and weight $\alpha$ verifies%
\[
K_{\lambda,\lambda}^{(k)}=1\text{ and }K_{\lambda,\alpha}^{(k)}\neq
0\Longrightarrow\alpha\leq\lambda
\]
where $\leq$ is the dominant order on partitions.

\begin{definition}
The $k$-Schur functions $s_{\kappa}^{(k)},\kappa\in\mathcal{B}_{k}$ are the
unique functions in $\Lambda_{(k)}$ such that
\[
h_{\delta}=\sum_{\delta\leq\kappa,\kappa\in\mathcal{B}_{k}}K_{\kappa,\delta
}^{(k)}s_{\kappa}^{(k)}%
\]
for any $\delta$ in $\mathcal{B}_{k}$.
\end{definition}

\begin{proposition}
[Pieri rule for $k$-Schur functions]For any $r\leq k$ and any $\kappa
\in\mathcal{B}_{k}$ we have
\begin{equation}
h_{r}s_{\kappa}^{(k)}=\sum_{\varkappa\in\mathcal{B}_{k}}s_{\varkappa}^{(k)}
\label{PieriKSchur}%
\end{equation}
where the sum is over all the $k$-bounded partitions $\varkappa$ such that
$\varkappa/\kappa$ is a weak horizontal strip of size $r$ in $\mathcal{B}_{k}$.
\end{proposition}

When $r=1$, the multiplication by $h_{1}$ is easily described by considering
all the possible $k$-bounded partitions at distance $1$ from $\kappa$ in
$\mathcal{B}_{k}$. Thanks to a geometric interpretation of the $k$-Schur
functions in terms of the homology of affine Grassmannians, Lam showed that
the product of two $k$-Schur functions is $k$-Schur positive:

\begin{theorem}
(\cite{LamLR}, Corollary 8.2) Given $\kappa$ and $\delta$ two $k$-bounded partitions, we have
\[
s_{\kappa}^{(k)}s_{\delta}^{(k)}=\sum_{\nu\in\mathcal{B}_{k}}c_{\lambda
,\delta}^{\nu(k)}s_{\nu}^{(k)}%
\]
with $c_{\lambda,\delta}^{\nu(k)}\in\mathbb{Z}_{\geq0}$.
\end{theorem}

\subsection{Some properties of $k$-Schur functions}

The $k$-conjugation operation $\omega_{k}$ can be read directly at the level
of $k$-bounded partitions without using the ordinary conjugation operations on
the $l$-cores (see (1.9)) in \cite{LLMSSZ}). To do this, start with a
$k$-bounded partition $\lambda=(\lambda_{1},\ldots,\lambda_{r})$ and decompose
it into its chains $\{c_{1},c_{2},\ldots,c_{r}\}$ where each chain is a
sequence of parts of $\lambda$ obtained recursively as follows. The procedure
is such that any part $\lambda_{i}$ is in the same chain as the part
$\lambda_{i+k-\lambda_{i}+1}$ when $i+k-\lambda_{i}+1\leq r$ (from the part
$\lambda_{i}$ one jumps $k-\lambda_{i}$ parts to get the following part of the
chain). Observe in particular that all the parts with length $k$ belong to the
same chain for in this case we jump $0$ parts. Once the chains $c_{i}$ are
determined, $\lambda^{\omega_{k}}$ is the partition with $k$-columns whose
lengths are the sums of the $c_{i}$'s$.$

\begin{example}
Consider the $5$-partition $\lambda=(\boldsymbol{5,5,5,4},$\textsl{4}%
$,\boldsymbol{3},$\textsl{3}$,3,\boldsymbol{2},$\textsl{2}$,1)$. Then, we get
$c_{1}=\{5,5,5,4,3,2\}$ next $c_{2}=\{4,3,2\},$ $c_{3}=\{3,1\},$
$c_{4}=\emptyset$ and $c_{5}=\emptyset$. So $\lambda^{\omega_{5}}$ is the
partition with columns of heights $24,$ $9$ and $4$.
\end{example}

The following facts will be useful.

\begin{enumerate}
\item Any partition $\lambda$ of rank at most $k$ is a $k$-bounded partition
and is then equal to its associated $l$-core (because $\lambda$ has no hook of
length $l=k+1$).

\item The lattice $\mathcal{B}_{k}$ coincides with the ordinary Young lattice
on the partitions of rank at most $k$. On this subset $\omega_{k}$ is the
ordinary conjugation.

\item For any partition $\lambda$ of rank at most $k,$ the $k$-Schur function
coincides with the ordinary Schur function that is $s_{\lambda}^{(k)}%
=s_{\lambda}$. In particular, the homogeneous functions $h_{1},\ldots,h_{k}$
and the elementary functions $e_{1},\ldots,e_{k}$ are the $k$-Schur functions
corresponding to the rows and columns partitions with at most $k$ boxes, respectively.
\end{enumerate}

The $k=2$ case is easily tractable because the lattice of $2$-bounded
partitions we consider has a simple structure. One verifies easily that for
any $2$-bounded partition $\lambda=(2^{a},1^{n-2a})$, we get in that case
\[
s_{\lambda}^{(2)}=\left\{
\begin{array}
[c]{l}%
h_{2}^{a}e_{2}^{\frac{n}{2}-a}\text{ when }n\text{ even,}\\
h_{2}^{a}e_{2}^{\frac{n-1}{2}-a}e_{1}\text{ when }n\text{ is odd.}%
\end{array}
\right.
\]
When $k>2,$ the structure of the graph $\mathcal{B}_{k}$ becomes more
complicated. Given a $k$-bounded partition $\lambda$ one can first precise
where it is possible to add a box in the Young diagram of $\lambda$ to get an
arrow in $\mathcal{B}_{k}$. Assume we add a box on the row $\lambda_{i}$ of
$\lambda$ to get the $k$-bounded partition $\mu$, denote by $c=\{\lambda
_{i_{1}},\ldots,\lambda_{i_{r}}\}$ the chain containing $\lambda_{i}$ where we
have $\lambda_{i}=\lambda_{i_{a}}$ with $a\in\{1,\ldots,r\}$.\ Observe we can
add components equal to zero to $c$ if needed since $\lambda$ is defined up to
an arbitrary number of zero parts. The following lemma permits to avoid the
use of $\omega_{k}$ in the construction of $\mathcal{B}_{k}$.

\begin{lemma}
\label{conditionArrow} There is an arrow $\lambda\rightarrow\mu$ in
$\mathcal{B}_{k}$ if and only if $\lambda_{i_{b}-1}=\lambda_{i_{b}}$ for any
$b=a+1,\ldots,r,$ that is if and only if each part located up to $\lambda_{i}$
in the chain containing $\lambda_{i}$ is preceded by a part with the same size.
\end{lemma}

\begin{proof}
One verifies that if the previous condition is not satisfied, $\mu^{\omega
_{k}}$ and $\lambda^{\omega_{k}}$ will differ by at least two boxes and if it
is satisfied by only one as desired.
\end{proof}

\bigskip

\begin{example}
\ 

\begin{enumerate}
\item One can always add a box on the first column of $\lambda$ since the
parts located up to the part $0$ are all equal to $0$.

\item Assume $k=2$, then we can add a box on the part $\lambda_{i}$ equal to
$1$ to get a part equal to $2$ if and only if there is an even number of parts
equal to $1$ up to $\lambda_{i}$. This is equivalent to says that $\lambda$
has an odd number of parts equal to $1$, that is that the rank of $\lambda$ is
odd since the other parts are equals to $2$ (or $0$).
\end{enumerate}
\end{example}

\bigskip

For any $a=1,\ldots,k$, let $R_{a}$ be the rectangle partition $(k-a+1)^{a}$.
The previous observations can be generalized (see \cite{LamLR} Corollary 8.3 and \cite{LLMSSZ} Section 4):

\begin{proposition}
\ 

\begin{enumerate}
\item Assume $\lambda$ is a $k$-bounded partition which is also a
$(k+1)$-core. Then $s_{\lambda}^{(k)}=s_{\lambda}$ (that is the $k$-Schur and
the Schur functions corresponding to $\lambda$ coincide).

\item In particular, for any rectangle partition $R_{a}$, we have $s_{R_{a}%
}^{(k)}=s_{R_{a}}$.

\item For any $a=1,\ldots,k$ and any $k$-bounded partition $\lambda$ we have
\[
s_{R_{a}}s_{\lambda}^{(k)}=s_{\lambda\cup R_{a}}^{(k)}%
\]
where $\lambda\cup R_{a}$ is obtained by adding $a$ times a part $k-a+1$ to
$\lambda$.
\end{enumerate}
\end{proposition}

\begin{corollary}
\label{reducedCase} For each $k$-bounded partition $\lambda$, there exists a
unique irreducible $k$-bounded partition\footnote{A $k$-bounded partition is
called irreducible when it contains less or equal to $a$ parts equal to $k-a$ for any
$a=0,\ldots,k-1$. Otherwise it is called reducible.} $\widetilde{\lambda}$ and a unique sequence of nonnegative
integers $p_{1},\ldots,p_{k}$ such that
\[
s_{\lambda}^{(k)}=\prod_{a=1}^{k}s_{R_{a}}^{p_{a}}s_{\widetilde{\lambda}%
}^{(k)}.
\]
In particular, the $k$-Schur functions are completely determined by the
$k$-Schur functions indexed by the irreducible $k$-bounded partitions and by
the ordinary Schur functions $s_{(k-a+1)^{a}}$, $a=1,\ldots,k$.
\end{corollary}

\begin{remark}
\label{Rem_bij}Write $\mathcal{P}_{\mathrm{irr}}$ for the set of irreducible
partitions. The map $\Delta:$ $\mathcal{B}_{k}\rightarrow\mathcal{P}%
_{\mathrm{irr}}\times\mathbb{Z}_{\geq0}^{k}$ which associates to each
$k$-bounded partition $\lambda$ the pair $(p,\widetilde{\lambda})$ where
$p=(p_{1},\ldots,p_{k})$ is a bijection.
\end{remark}

\begin{example}
Assume $k=4$ and $\lambda=(4,4,3,3,3,3,3,2,2,2,2,1,1,1,1,1)$. Then we get
$\widetilde{\lambda}=(3,2,1)$ and
\[
s_{\lambda}^{(k)}=s_{(4)}^{2}s_{(3,3)}^{2}s_{(2,2,2)}s_{(1,1,1,1)}%
s_{\widetilde{\lambda}}^{(k)}.
\]

\end{example}

We conclude this paragraph by recalling other important properties of
$k$-Schur functions.\ We have first the inclusions of algebras $\Lambda
_{(k)}\subset\Lambda_{(k+1)}\subset\Lambda$.

\begin{proposition}
(see \cite{LLMSSZ} Section 4)\ \label{Prop_positiveExp}

\begin{enumerate}
\item Each $k$-Schur function has a positive expansion on the basis of
$(k+1)$-Schur functions.

\item Each $k$-Schur function has a positive expansion on the basis of
ordinary Schur functions.
\end{enumerate}
\end{proposition}

\subsection{Harmonic functions and minimal boundary of $\mathcal{B}_{k}$}

\begin{definition}
A function $f:\mathcal{B}_{k}\rightarrow\mathbb{R}$ is said harmonic when
\[
f(\lambda)=\sum_{\lambda\rightarrow\mu}f(\mu)\text{ for any }\lambda
\in\mathcal{B}_{k}.
\]
We denote by $\mathcal{H}(\mathcal{B}_{k})$ the set of harmonic functions on
$\mathcal{B}_{k}$.
\end{definition}

Another way to understand harmonic functions is to introduce the infinite
matrix $\mathcal{M}$ of the graph $\mathcal{B}_{k}$. The harmonic functions on
$\mathcal{B}_{k}$ then correspond to the right eigenvectors for $\mathcal{M}$
associated to the eigenvalue $1$. One can also consider $t$-harmonic functions
which correspond to the right eigenvectors for $\mathcal{M}$ associated to the
eigenvalue $t$. Clearly $\mathcal{H}(\mathcal{B}_{k})$ is a vector space over
$\mathbb{R}$. In fact, we mostly restrict ourself to the set $\mathcal{H}%
^{+}(\mathcal{B}_{k})$ of positive harmonic functions for which $f$ takes
values in $\mathbb{R}_{\geq0}$. Then, $\mathcal{H}^{+}(\mathcal{B}_{k})$ is a
cone since it is stable by addition and multiplication by a positive real. To
study $\mathcal{H}^{+}(\mathcal{B}_{k})$, we only have to consider its subset
$\mathcal{H}_{1}^{+}(\mathcal{B}_{k})$ of normalized harmonic functions such
that $f(1)=1$. In fact, $\mathcal{H}_{1}^{+}(\mathcal{B}_{k})$ is a convex set
and its structure is controlled by its extremal subset $\partial
\mathcal{H}^{+}(\mathcal{B}_{k})$.

\noindent We aim to characterize the extremal positive harmonic functions
defined on $\mathcal{B}_{k}$ and obtain a simple parametrization of
$\partial\mathcal{H}^{+}(\mathcal{B}_{k})$. By using the Pieri rule on
$k$-Schur functions, we get%
\[
s_{\lambda}s_{(1)}=\sum_{\lambda\rightarrow\mu}s_{\mu}%
\]
for any $k$-bounded partitions $\lambda$ and $\mu$.\ This means that
$\mathcal{B}_{k}$ is a so-called multiplicative graph with associated algebra
$\Lambda_{(k)}$. Moreover, if we denote by $K$ the positive cone spanned by
the set of $k$-Schur functions, we can apply the ring theorem of Kerov and
Vershik (see for example \cite[Section 8.4]{LT}) which characterizes the
subset of extreme points $\partial\mathcal{H}^{+}(\mathcal{B}_{k})$.\ Denote
by $\mathrm{Mult}^{+}(\Lambda_{(k)})\subset\Lambda_{(k)}^{\ast}$ the set of
multiplicative functions on $\Lambda_{(k)}$ which are nonnegative on $K$ and
equal to $1$ on $s_{1}$. So a function $f$ belongs to $\mathrm{Mult}%
^{+}(\Lambda_{(k)})$ when $f:\Lambda_{(k)}\rightarrow\mathbb{R}$ is linear and
satisfies $f(UV)=f(U)f(V)$ for any $U,V\in\Lambda_{(k)}$. Note that
$i:\mathcal{B}_{k}\longrightarrow\Lambda_{(k)}$ such that $i(\lambda
)=s_{\lambda}^{(k)}$ induces a map $i^{\ast}:\Lambda_{(k)}^{\ast
}\longrightarrow F(\mathcal{B}_{k},\mathbb{R})$. Let $K.K$ be the set of products of two elements in $K$. Since we have $K.K\subset K$,
we get the following algebraic characterization of $\partial\mathcal{H}%
^{+}(\mathcal{B}_{k})$.

\begin{proposition}
\label{multiGraphExtreme}The map $i^{\ast}$ yields an homeomorphism between
$\mathrm{Mult}^{+}(\Lambda_{(k)})$ and $\partial\mathcal{H}^{+}(\mathcal{B}%
_{k})$.\ 
\end{proposition}

Since $i(\mathcal{B}_{k})$ is a basis of $\Lambda_{(k)}$, this means that
$\partial\mathcal{H}(\mathcal{B}_{k})$ is completely determined by the
$\mathbb{R}$-algebra morphisms $\varphi:\Lambda_{(k)}\rightarrow\mathbb{R}$
such that $\varphi(s_{1})=1$ and $\varphi(s_{\lambda}^{(k)})\geq0$ for any
$k$-bounded partition $\lambda$. Each function $f\in\partial\mathcal{H}%
^{+}(\mathcal{B}_{k})$ can then be written $f=\varphi\circ i$.

By Corollary \ref{reducedCase}, the condition $\varphi(s_{\lambda}^{(k)}%
)\geq0$ for each $k$-bounded partition reduces in fact to test a finite number
of $k$-Schur functions, namely $\varphi(s_{\widetilde{\lambda}}^{(k)})\geq0$
for each irreducible $k$-bounded partition (there are $k!$ such
partitions) and $\varphi(s_{(k-a+1)^{a}})\geq0$ for any $a=1,\ldots,k$. We
will see that the condition $\varphi(s_{\lambda}^{(k)})>0$ for each
$k$-bounded partition reduces in fact to test the positivity on the Schur
functions associated to partitions with maximal hook-length less or equal to
$k$ (See Remark \ref{Rema_redtest}).

\section{Restricted graph and irreducibility}

\subsection{The matrix $\Phi$}

\label{subsec_MatrixFI}By Corollary \ref{reducedCase}, each morphism
$\varphi:\Lambda_{(k)}\rightarrow\mathbb{R}$ is uniquely determined by its
values on the rectangle Schur functions $s_{R_{a}},1\leq a\leq k$ and on each
$s_{\tilde{\lambda}}^{(k)}$ where $\tilde{\lambda}$ is an irreducible
$k$-bounded partition. Set $r_{a}=\varphi(s_{R_{a}}),a=1,\ldots,k$ and
$\vec{r}=(r_{1},\ldots,r_{k})$. Recall that $\mathcal{P}_{\mathrm{irr}}$ is
the set of irreducible $k$-bounded partitions (including the empty partition).
Then, for $\lambda\in\mathcal{P}_{\mathrm{irr}}$,
\begin{equation}
\varphi(s_{\lambda}^{(k)})\varphi(s_{(1)})=\sum_{\lambda\rightarrow\mu}%
\varphi(s_{\mu}^{(k)}). \label{PieriPhi}%
\end{equation}
By Corollary \ref{reducedCase}, for each $k$-bounded partition $\mu$ there
exists a sequence $\{p_{1}^{\mu},p_{2}^{\mu}\ldots,p_{k}^{\mu}\}$ of elements
in $\{1,\dots,k\}$ and an irreducible partition $\widetilde{\mu}$ such that
\begin{equation}
s_{\mu}^{(k)}=\prod_{a=1}^{k}s_{R_{a}}^{p_{a}^{\mu}}s_{\widetilde{\mu}}%
^{(k)}\text{ and thus }\varphi(s_{\mu}^{(k)})=\prod_{a=1}^{k}r_{a}^{p_{a}%
^{\mu}}\varphi(s_{\widetilde{\mu}}^{(k)}). \label{decompos}%
\end{equation}
Hence by setting
\[
\varphi_{\lambda\nu}=\sum_{\substack{\lambda\rightarrow\mu\\\widetilde{\mu
}=\nu}}\prod_{1\leq a\leq k}r_{a}^{p_{a}^{\mu}}%
\]
we get
\[
\varphi(s_{\lambda}^{(k)})=\sum_{\nu\in\mathcal{P}_{\mathrm{irr}}}%
\varphi_{\lambda\nu}\varphi(s_{\nu}^{(k)}).
\]
Let $\Phi_{(r_{1},\ldots,r_{k})}:=(\varphi_{\nu\lambda})_{\lambda,\nu
\in\mathcal{P}_{\mathrm{irr}}}$ \footnote{Observe we have defined
$\Phi_{(r_{1},\ldots,r_{k})}$ as the transpose of the matrix $(\varphi
_{\lambda,\mu})_{\lambda,\nu\in\mathcal{P}_{\mathrm{irr}}}$ to make it
compatible with the multiplication by $s_{(1)}$ used in Section
\ref{Secion_Field}.} and define $f\in\mathbb{R}^{\mathcal{P}_{\mathrm{irr}}}$
as the vector $(\varphi(s_{\lambda}^{(k)}))_{\lambda\in\mathcal{P}%
_{\mathrm{irr}}}$. When there is no risk of confusion, we simply write $\Phi$
instead of $\Phi_{(r_{1},\ldots,r_{k})}$. The vector $f$ is a left eigenvector
of $\Phi$ for the eigenvalue $\varphi(s_{1})$ with positive entries having
value $1$ on $\emptyset$ and $\varphi(s_{1})$ on $s_{1}$.

\subsection{Irreducibility of the matrix $\Phi$}

Recall that a matrix $M\in M_{n}(\mathbb{R})$ with nonnegative entries is
irreducible if and only if for each $1\leq i,j\leq n$ there exists $n\geq1$
such that $(M^{n})_{ij}>0$.

\begin{proposition}
\label{irredPhi}Consider $\vec{r}=(r_{1},\ldots,r_{k})\in\mathbb{R}_{\geq
0}^{k}.$ Then, the matrices $\Phi_{(r_{1},\ldots,r_{k})}$ and $\Phi
_{(r_{1},\ldots,r_{k})}^{t}$ associated to $\varphi$ are irreducible if and
only if for all $1\leq a\leq k-1$, $r_{a}$ or $r_{a+1}$ is positive.
\end{proposition}

We will prove in fact that $\Phi^{t}$ is irreducible.\ Let $G$ be the graph
with set of vertices $\mathcal{P}_{\mathrm{irr}}$ and a directed edge from
$\lambda$ to $\nu$ if and only if $\Phi_{\lambda\nu}\not =0$. The matrix
$\Phi^{t}$ is irreducible if and only if $G$ is strongly connected, which
means that there is a (directed) path from any vertex to any other vertex of
the graph. We prove Proposition \ref{irredPhi} by showing that $G$ is strongly
connected. Let us first establish a preliminary lemma. We say that $\lambda
\in\mathcal{P}_{\mathrm{irr}}$ is $1$-saturated when it contains less than
$k-1$ parts equal to $1$. More generally, for $i\geq2$ the partition
$\lambda\in\mathcal{P}_{\mathrm{irr}}$ is $i$-saturated when we have%
\[
\lambda=(\ldots,(i-1)^{k-i+1},\dots,2^{k-2},1^{k-1})\text{ and }\lambda
\neq(\ldots,i^{k-i},(i-1)^{k-i+1},\dots,2^{k-2},1^{k-1}).
\]
Denote by $\lambda^{1}$ the irreducible $k$-bounded partition $((k-1)^{1}%
,(k-2)^{2},\dots,1^{k-1})$. Observe that $\lambda\in\mathcal{P}_{\mathrm{irr}%
}$ is $k$-saturated if and only if $\lambda=\lambda^{1}$.

\begin{lemma}
Any vertex $\lambda$ of $G$ is connected to $\lambda^{1}$.
\end{lemma}

\begin{proof}
We get a path between $\lambda$ and $\lambda^{1}$ from the following observations.

\noindent1: If $\lambda$ is $i$-saturated, we claim that $\lambda
\rightarrow\lambda^{\uparrow i-1}$, where $\lambda^{\uparrow i-1}$ is the
partition obtained by adding one box to the first row of size $(i-1)$ of
$\lambda$. Moreover, $\lambda^{\uparrow i-1}$ is then $(i-1)$-saturated. To
see this, consider the minimal integer $r$ such that $\lambda_{r}=i-1$ and the
chain $c$ associated to $\lambda_{r}$. Since $\lambda$ is $i$-saturated, a
combinatorial computation shows that
\[
c=\{\ldots,\lambda_{r},\lambda_{r+k-i+2},\lambda_{r+(k-i+2)+(k-i+3)}%
,\ldots,\lambda_{m}\}\text{ with }m=r+\sum_{s=0}^{i-1}(k-i+s).
\]
Moreover, for $0\leq t\leq i-1$, the $(t+1)$-th row in the chain $c$ is
$\lambda_{r+\sum_{s=0}^{t}(k-i+s)}$ and has length $i-1-t$. In particular,
each row following $\lambda_{r}$ in the chain $c$ is preceded by a row with
the same length. Therefore, we can apply Lemma \ref{conditionArrow} to get the
existence of an edge between $\lambda$ and the partition $\lambda^{\uparrow
i-1}$ which is $(i-1)$-saturated.

\noindent2: Suppose that $\lambda$ is $i$-saturated. Applying successively the
previous procedure on the partitions $\lambda,\lambda^{\uparrow(i-1)}%
,(\lambda^{\uparrow(i-1)})^{\uparrow(i-2)},\dots$ eventually yields a
partition $\nu$ which is irreducible, $i$-saturated, with one more row of
length $i$ than $\lambda$ and such that there is a path between $\lambda$ and
$\nu$ in $G$.

\noindent3: Suppose that $\lambda$ has initially $l\leq k-i-1$ rows of size
$i$. By repeating $k-i-l$ times step 2, one gets a partition $\kappa$
connected to $\lambda$ in $\Gamma$ which is $(i+1)$-saturated.

\noindent4: Repeated applications of step 3 for $i<l\leq k$ yields a path
between $\lambda$ and $\lambda^{1}$ in $G$.
\end{proof}

We prove Proposition \ref{irredPhi} by giving necessary and sufficient
conditions to have a path on $G$ between $\lambda^{1}$ and $\emptyset$.

\begin{proof}
[Proof of Proposition \ref{irredPhi}]Let us show that there is a path between
$\lambda^{1}$ and $\emptyset$ if and only if for all $1\leq a\leq k-1$,
$\varphi(s_{R_{a}})$ or $\varphi(s_{R_{a+1}})$ is positive.

Assume first that for all $1\leq a\leq k-1$, $\varphi(s_{R_{a}})$ or
$\varphi(s_{R_{a+1}})$ is positive.

\noindent For $1\leq a\leq k-1$, let $\lambda^{a}$ be the partition
$(k-1,(k-2)^{2},\ldots,a^{k-a})$, and set $\lambda^{k}=\emptyset$. Let us
prove that for $1\leq a\leq k-1$, there is a path on $G$ from $\lambda^{a}$ to
$\lambda^{a+1}$. If $\varphi(s_{R_{a}})>0$, it suffices to add a part of
length $a$ at the end of $\lambda^{a}$, which is always possible. This gives a
rectangle $a^{k-a+1}$ which, once removed, yields $\lambda^{a+1}$. Assume now
we have $\varphi(s_{R_{a}})=0$ and thus $\varphi(s_{R_{a+1}})>0$ by the
hypothesis on $\varphi$. Let $i$ be minimal such that $\lambda_{i}^{a}=a$, and
let $c$ be the chain containing $\lambda_{i}^{a}$. On the one hand, since
$\lambda_{i}^{a}=a$, the part following $\lambda_{i}^{a}$ in $c$ is
$\lambda_{i+k-a+1}^{a}$. On the other hand, by definition of $\lambda^{a}$, we
have $\lambda_{i+(k-a)}^{a}=\lambda_{i+(k-a)+1}^{a}=0$. Thus, by Lemma
\ref{conditionArrow}, we can add a box on the part $\lambda_{i}^{a}$, which
makes appear a block $(a+1)^{k-a+1}$. Since $\varphi(s_{R_{a+1}})>0$ we so get
an arrow in $G$ between $\lambda^{a}$ and the partition $(\lambda
^{a-2},\lambda_{i+1}^{a},\ldots,\lambda_{i+(k-a)-1}^{a})$. Similarly, we can
successively add a box to the parts $\lambda_{i+1}^{a},\ldots,\lambda
_{i+(k-a)-1}^{a}$ (all equal to $a$) and get a path between $(\lambda
^{a-2},\lambda_{i+1}^{a},\ldots,\lambda_{i+(k-a)-1}^{a})$ and $\lambda^{a+1}$.
We so obtain a path in $G$ from $\lambda^{a}$ to $\lambda^{a+1}$ for all
$1\leq a\leq k-1$ and thus a path from $\lambda^{1}$ to $\lambda^{k}%
=\emptyset$.

Assume now that there exists $1\leq a\leq k-1$ such that $\varphi(s_{R_{a}%
})=\varphi(s_{R_{a+1}})=0$.

\noindent Let $\gamma=(\mu^{1},\ldots,\mu^{r})$ be a path on $G$ starting at
$\mu^{1}=a^{k-a}$, and denote by $x_{i}$ and $y_{i}$ the number of parts in
$\mu^{i}$ equal to $a+1$ and $a$, respectively. Since $\gamma$ is a path on
$G$, for all $1\leq i\leq r$, we have $x_{i}\leq k-a-1$ and $y_{i}\leq k-a$.
Let us prove by induction on $1\leq i\leq r$ that $x_{i}+y_{i}\geq k-a$.

\noindent This is certainly true for $i=1$. Assume that $i>1$ and the result
holds for $i-1$. Since $\varphi(s_{R_{a}})=\varphi(s_{R_{a+1}})=0$, the only
way to get $x_{i}+y_{i}<x_{i-1}+y_{i-1}$ is to add a box in the first part of
length $a+1$ (if any) of $\mu^{i-1}$. So if we have $x_{i-1}=0$, then
$x_{i}+y_{i}\geq x_{i-1}+y_{i-1}\geq k-a$. Assume now that $x_{i-1}>0$. Let
$l$ be minimal such that $\mu_{l}^{i-1}=a+1$ and let $c$ be the chain
containing $\mu_{l}^{i-1}$. The part following $\mu_{l}^{i-1}$ in $c$ is then
equal to $l+k-(a+1)+1=l+(k-a)$. Since $x_{i-1}+y_{i-1}\geq k-a$ and
$x_{i-1}\leq k-a-1$, we have $\mu_{l+(k-a)-1}^{i-1}=a$. Thus, by Lemma
\ref{conditionArrow}, it is possible to add a box on the part $l$ of
$\mu^{i-1}$ if and only if $\mu_{l+k-a}^{i-1}=a.$ If so, we get in fact
$x_{i-1}+y_{i-1}\geq(l+k-a)-l+1\geq k-a+1$ and $x_{i}+y_{i}=x_{i-1}+y_{i-1}%
-1$. Therefore, in any cases, $x_{i}+y_{i}\geq k-a$. We have proved that for
any path on $G$ starting at $a^{k-a}$, the number of parts of length $a$ or
$a+1$ is at least $k-a$, thus there is no path between $a^{k-a}$ and
$\emptyset$ in $G$.
\end{proof}

\section{Field extensions and $k$-Schur functions}

\label{Secion_Field}

\subsection{Field extensions}

Recall that $\Lambda_{(k)}\mathbb{=R}[h_{1},\ldots,h_{k}]$.\ Since
$h_{1},\ldots,h_{k}$ are algebraically independent over $\mathbb{R}$, we can
consider the fraction field $\mathbb{L}=\mathbb{R}(h_{1},\ldots,h_{k}%
)$.\ Write $\mathbb{A}=\mathbb{R}[s_{R_{1}},\ldots s_{R_{k}}]$ the subalgebra
of $\Lambda_{(k)}$ generated by the rectangle Schur functions $s_{R_{a}%
},a=1,\ldots,k$. In order to introduce the fraction field of $\mathbb{A}$, we
first need to check that $s_{R_{1}},\ldots s_{R_{k}}$ are algebraically
independent over $\mathbb{R}$. We shall use a proposition giving a sufficient
condition on a family of polynomials to be algebraically independent.

Let $k$ be a field and $k[T_{1},\ldots,T_{m}]$ the ring of polynomials in
$T_{1},\ldots,T_{m}$ over $k$.\ For any $\beta\in\mathbb{Z}_{\geq0}^{m}$, we
set $T^{\beta}=T_{1}^{\beta_{1}}\cdots T_{m}^{\beta_{m}}$. We also consider a total order $\preceq$ on the monomials of $k[T_{1},\ldots,T_{m}]$ given by the lexicographical order on the exponents. Namely, $T^{\beta}\succeq T^{\beta'}$ if and only if $\beta_{i}> \beta'_{i}$ on the first index on which $\beta$ and $\beta'$ differ, if any. The
leading monomial $\mathrm{lm}(P)$ of a polynomial $P\in k[T_{1},\ldots,T_{m}]$
is the monomial appearing in the support of $P$ (that is, with a nonzero
coefficient) maximal under the total order $\preceq$.

\begin{proposition}
(See \cite[Lemma 4.2.10]{Mitt}, \cite[Proposition 6.6.11]{KR})\label{Prop_alge_inde}

\begin{enumerate}
\item The monomials $T^{\beta^{(1)}},\ldots,T^{\beta^{(p)}}$ are algebraically
independent if and only if $\beta^{(1)},\ldots,\beta^{(p)}$ are linearly
independent over $\mathbb{Z}$.

\item Consider $P_{1},\ldots,P_{l}$ polynomials in $k[T_{1},\ldots,T_{m}]$
such that $\mathrm{lm}(P_{1}),\ldots,\mathrm{lm}(P_{l})$ are algebraically
independent, then $P_{1},\ldots,P_{l}$ are algebraically independent.
\end{enumerate}
\end{proposition}
We also refer to \cite[Section 3]{Pan} for a detailed proof of this proposition. With Proposition \ref{Prop_alge_inde} in hand, it is then easy to check that
$s_{R_{1}},\ldots s_{R_{k}}$ are algebraically independent over $\mathbb{R}$.
Recall that each Schur function $s_{\lambda}$ with indeterminate set
$X=\{X_{1},X_{2}\ldots\}$ decomposes on the form%
\[
s_{\lambda}=X^{\lambda}+\sum_{\mu<\lambda}K_{\lambda,\mu}X^{\mu}%
\]
where $\leq$ is the dominant order over finite sequences of integers, that is,
$\beta<\beta^{\prime}$ when $\beta^{\prime}-\beta$ decomposes as a sum of
$\varepsilon_{i}-\varepsilon_{j},i<j$ with nonnegative integer coefficients.
Since the lexicographical order refines the dominance order,
by Assertion 1 of the previous proposition, the monomials $X^{R_{1}}%
,\ldots,X^{R_{k}}$ are algebraically independent since the rectangle
partitions $R_{a},a=1,\ldots,k$ are linearly independent over $\mathbb{Z}$.
Assertion 2 then implies that $s_{R_{1}},\ldots s_{R_{k}}$ are algebraically
independent over $\mathbb{R}$. We denote by $\mathbb{K}=\mathbb{R}(s_{R_{1}%
},\ldots s_{R_{k}})$ the fraction field of the algebra $\mathbb{A}$.

\begin{proposition}
\ \label{PropLK}

\begin{enumerate}
\item Each $h_{1},\ldots,h_{k}$ is algebraic over $\mathbb{K}$.

\item We have $\mathbb{L}=\mathbb{R}(h_{1},\ldots,h_{k})=\mathbb{K}%
[h_{1},\ldots,h_{k}]$.

\item $\mathbb{L}$ is an algebraic extension of $\mathbb{K}$.

\item The field $\mathbb{L}$ is a finite extension of $\mathbb{K}$ with degree
$[\mathbb{L}:\mathbb{K}]=k!$ and the set $\mathcal{I}=\{s_{\kappa}^{(k)}%
\mid\kappa\in\mathcal{P}_{\mathrm{irr}}\}$ is a basis of $\mathbb{L}$ over
$\mathbb{K}$.

\item $\Lambda_{(k)}$ is an integral extension of $\mathbb{A}$.
\end{enumerate}
\end{proposition}

\begin{proof}
1: For any $a=1,\ldots,k$, consider the evaluation morphism $\theta
_{a}:\mathbb{K}[T]\rightarrow\mathbb{L}$ which associates to any
$P\in\mathbb{K}[T]$, the polynomial $P(h_{a})\in\Lambda_{(k)}\mathbb{=R}%
(h_{1},\ldots,h_{k})$.\ We know that $\{s_{\lambda}^{(k)}\mid\lambda
\in\mathcal{B}_{k}\}$ is a basis of $\Lambda_{(k)}$ over $\mathbb{R}$ thus
each power $h_{a}^{i},i\in\mathbb{Z}_{\geq0}$ decomposes on the basis of
$k$-Schur functions with real coefficients.\ By Corollary \ref{reducedCase},
$h_{a}^{i}$ then decomposes on the family $\mathcal{I}$ with coefficients in
$\mathbb{K}$. Thus, $P(h_{a})$ also decomposes on the family $\mathcal{I}$
with coefficients in $\mathbb{K}$. Since $\mathcal{I}$ is a finite set, this
shows that $\mathrm{Im}(\theta_{a})$ is a finite-dimensional $\mathbb{K}%
$-subspace of $\mathbb{L}$, thus $h_{a}$ is algebraic over $\mathbb{K}$.

2: Since $h_{1},\ldots,h_{k}$ are algebraic over $\mathbb{K}$, we get
$\mathbb{K}[h_{1},\ldots,h_{k}]=\mathbb{K}(h_{1},\ldots,h_{k})\subset
\mathbb{L}$. We also have $\mathbb{L}=\mathbb{R}(h_{1},\ldots,h_{k}%
)\subset\mathbb{K}(h_{1},\ldots,h_{k})$ since $\mathbb{K}$ is an extension of
$\mathbb{R}$. Thus, $\mathbb{L}=\mathbb{K}[h_{1},\ldots,h_{k}]$.

3: This easily follows from 1 and 3$.$

4: By using the same arguments as in the proof of 1, we get that each element
$Q(h_{1},\ldots,h_{k})$ with $Q\in\mathbb{K}[T_{1},\ldots,T_{k}]$ decomposes
on the family $\mathcal{I}$ with coefficients in $\mathbb{K}$. Assume we have
\[
\sum_{\kappa\mid k\text{-irreducible}}c_{\kappa}s_{\kappa}^{(k)}=0
\]
with $c_{\kappa}\in\mathbb{K}=\mathbb{R}(s_{R_{1}},\ldots,s_{R_{k}})$ for any
$k$-irreducible partition $\kappa$. Up to multiplication, we can assume these
coefficients belong in fact to $\mathbb{R}[s_{R_{1}},\ldots,s_{R_{k}}]$. Set
\[
c_{\kappa}=\sum_{\beta\in\mathbb{Z}_{\geq0}^{k}}a_{\beta}^{(\kappa)}s_{R_{1}%
}^{\beta_{1}}\cdots s_{R_{k}}^{\beta_{k}}%
\]
where all the coefficients $a_{\beta}^{(\kappa)}$ are equal to $0$ up to a
finite number (in which case $a_{\beta}^{(\kappa)}$ is real). We get%
\begin{equation}
\sum_{\kappa\mid k\text{-irreducible}}\sum_{\beta\in\mathbb{Z}_{\geq0}^{k}%
}a_{\beta}^{(\kappa)}s_{R_{1}}^{\beta_{1}}\cdots s_{R_{k}}^{\beta_{k}%
}s_{\kappa}^{(k)}=0. \label{CL}%
\end{equation}
By Remark \ref{Rem_bij}, there is a bijection between the set of $k$-bounded
partitions and that of pairs $(\beta,\kappa)$ with $\kappa$ $k$-irreducible
and $\beta\in\mathbb{Z}_{\geq0}^{k}$.\ So equation (\ref{CL}) gives in fact a
linear combination of $k$-Schur functions over $\mathbb{R}$ equal to $0$.
Since we know that the set of $k$-Schur functions is a basis of $\Lambda
_{(k)}$, this imposes that each coefficient $a_{\beta}^{(\kappa)}$ is in fact
equal to $0$.\ So the family $\mathcal{I}$ is a $\mathbb{K}$-basis of
$\mathbb{L}$ and we have $[\mathbb{L}:\mathbb{K}]=\mathrm{card}(\mathcal{I}%
)=k!.$

5: The characteristic polynomial of each $h_{a},a=1,\ldots,k$ belongs to
$\mathbb{A}[T]$ because the multiplication by $h_{a}$ on the basis
$\mathcal{I}$ makes only appear coefficients in $\mathbb{A}$. Thus, each
$h_{a}$ is an integral element of $\Lambda_{(k)}=\mathbb{A}[h_{1},\ldots
,h_{k}]$ over $\mathbb{A}.$
\end{proof}

\subsection{Primitive element}

\label{subsec_primitiveel}By Proposition \ref{PropLK}, $s_{1}=h_{1}$ is
algebraic over $\mathbb{K}$. Denote by $\Pi$ its minimal polynomial.\ Observe
that $\Pi$ is an irreducible polynomial of $\mathbb{K}[T]$. Also write
$\boldsymbol{\Phi}$ for the matrix of the multiplication by $s_{1}$ in
$\mathbb{L}$ in the basis $\mathcal{I}=\{s_{\kappa}^{(k)}\mid\kappa
\in\mathcal{P}_{\mathrm{irr}}\}$ (here, we assume we have fixed once for all a
total order on the set $\mathcal{P}_{\mathrm{irr}}$ of $k$-irreducible
partitions).\ Let $\Xi(T)=\mathrm{det}(TI_{k!}-\boldsymbol{\Phi})$ be the
characteristic polynomial of the matrix $\boldsymbol{\Phi}$. We thus have that
$\Pi$ divides $\Xi$. Moreover both polynomials belong to $\mathbb{A}[T]$ since
the entries of the matrix $\boldsymbol{\Phi}$ are in the ring $\mathbb{A}$. We
have in fact the following stronger proposition.

\begin{proposition}
\ \label{Prop_InvSim}

\begin{enumerate}
\item The invariant factors of the multiplication by $s_{1}$ are all equal to
$\Pi$: there exists an integer $m$ such that $\Xi=\Pi^{m}.$

\item The coefficients of $\Pi$ and $\Xi$ are invariant under the flips of
$s_{R_{a}}$ and $s_{R_{k-a+1}}$ for any $a=1,\ldots,\left\lfloor \frac{k}%
{2}\right\rfloor $.
\end{enumerate}
\end{proposition}

\begin{proof}
1: Write $P_{1},\ldots,P_{m}$ for the invariant factors of $\boldsymbol{\Phi}%
$. We must have $P_{1}/P_{2}/\cdots/P_{m},$ $P_{m}=\Pi$ and $P_{1}P_{2}\cdots
P_{m}=\Xi$. Since $\Pi$ is irreducible, this imposes $P_{1}=\cdots=P_{m}=\Pi$
which gives the assertion.

2: We apply $\omega_{k}$ to each equality $\Pi(s_{(1)})=0$ and $\Xi
(s_{(1)})=0$.
\end{proof}

\bigskip

\begin{example}
\label{ex_k3_Phi}For $k=3,$ we get by listing the $k$-bounded partitions as
$\emptyset,(1),(2),(1,1),(2,1)$ and $(2,1,1)$%
\[
\boldsymbol{\Phi}=\left(
\begin{array}
[c]{cccccc}%
0 & 0 & s_{R_{1}} & s_{R_{3}} & s_{R_{2}} & 0\\
1 & 0 & 0 & 0 & 0 & s_{R_{2}}\\
0 & 1 & 0 & 0 & 0 & s_{R_{3}}\\
0 & 1 & 0 & 0 & 0 & s_{R_{1}}\\
0 & 0 & 1 & 1 & 0 & 0\\
0 & 0 & 0 & 0 & 1 & 0
\end{array}
\right)
\]
This gives $\Xi(T)=T^{6}-2\left(  s_{R_{1}}+s_{R_{3}}\right)  T^{3}-4s_{R_{2}%
}T^{2}+\left(  s_{R_{1}}-s_{R_{3}}\right)  ^{2}.$ Observe the symmetry of
$\boldsymbol{\Phi}$ which will be elucidated in \S \ \ref{SubSec_Invol}.
\end{example}

In Proposition \ref{Prop_InvSim}, we have just used that $\boldsymbol{\Phi}$
is the matrix of the multiplication by $s_{(1)}$ which is algebraic over
$\mathbb{K}$. Given any morphism $\widetilde{\varphi}:\mathbb{A}%
\rightarrow\mathbb{R}$ such that $\widetilde{\varphi}(s_{R_{a}})\geq0$ for any
$a=1,\ldots,k$, the matrix $\widetilde{\varphi}(\boldsymbol{\Phi})$ obtained
by replacing in $\boldsymbol{\Phi}$ each rectangle Schur function $s_{R_{a}}$
by $\widetilde{\varphi}(s_{R_{a}})$ coincides with the matrix $\Phi$ defined
in \S \ \ref{subsec_MatrixFI}.\ Thus $\Phi=\widetilde{\varphi}%
(\boldsymbol{\Phi})$ has nonnegative entries.\ We can now state the
main theorem of this section.

\begin{theorem}
We have $\mathbb{L}=\mathbb{K}(s_{(1)})$, that is $s_{(1)}$ is a primitive
element for $\mathbb{L}$ regarded as an extension of $\mathbb{K}$.
\end{theorem}

\begin{proof}
It suffices to show that $\Pi=\Xi$.\ By Proposition \ref{Prop_InvSim}, we
already know that $\Xi=\Pi^{m}$ with $m\in\mathbb{Z}_{>0}$. Then by Frobenius
reduction, there exists an invertible matrix $P$ with entries in $\mathbb{K}$
such that
\begin{equation}
\boldsymbol{\Phi}=P\left(
\begin{array}
[c]{ccc}%
\mathcal{C}_{\Pi} & \boldsymbol{0} & \boldsymbol{0}\\
\boldsymbol{0} & \ddots & \boldsymbol{0}\\
\boldsymbol{0} & \boldsymbol{0} & \mathcal{C}_{\Pi}%
\end{array}
\right)  P^{-1}, \label{FIB}%
\end{equation}
that is, the matrix $\boldsymbol{\Phi}$ is equivalent to a block diagonal
matrix with $k$ blocks equal to $\mathcal{C}_{\Pi}$ the companion matrix of
the polynomial $\Pi$. By multiplying the columns of the matrix $P$ by elements
of $\mathbb{A}$, one can also assume that the entries of $P$ belong to
$\mathbb{A}$.\ Then we can write $P^{-1}=\frac{1}{\det(P)}Q$ where $Q$ has
also entries in $\mathbb{A}$ and $\det(P)\in\mathbb{A}$ is nonzero. Since
$\det(P)\in\mathbb{A}=\mathbb{R}[s_{R_{1}},\ldots,s_{R_{k}}]$ is nonzero,
there exists a nonzero polynomial $F\in\mathbb{R}[T_{1},\ldots,T_{k}]$ such
that $\det(P)=F(s_{R_{1}},\ldots,s_{R_{k}})$.\ Also a morphism $\widetilde
{\varphi}:\mathbb{A}\rightarrow\mathbb{R}$ such that $\widetilde{\varphi
}(s_{R_{a}})\geq0$ for any $a=1,\ldots,k$ is characterized by the datum of the
$\widetilde{\varphi}(s_{R_{a}})$'s. The polynomial $F$ being nonzero, one can
find $(r_{1},\ldots,r_{k})\in\mathbb{R}_{>0}^{k}$ such that $F(r_{1}%
,\ldots,r_{k})\neq0$.\ For such a $k$-tuple, let us define $\widetilde
{\varphi}$ by setting $\widetilde{\varphi}(s_{R_{a}})=r_{a}$. Then
$\widetilde{\varphi}(\det(P))\neq0$ and we can apply $\widetilde{\varphi}$ to
(\ref{FIB}) which gives%
\[
\Phi=\widetilde{\varphi}(P)\left(
\begin{array}
[c]{ccc}%
\mathcal{C}_{\widetilde{\varphi}(\Pi)} & \boldsymbol{0} & \boldsymbol{0}\\
\boldsymbol{0} & \ddots & \boldsymbol{0}\\
\boldsymbol{0} & \boldsymbol{0} & \mathcal{C}_{\widetilde{\varphi}(\Pi)}%
\end{array}
\right)  \widetilde{\varphi}(P)^{-1}.
\]
The matrix $\Phi$ has nonnegative entries and is irreducible by Lemma
\ref{irredPhi}.\ So, by Perron Frobenius theorem, it admits a unique
eigenvalue $t>0$ of maximal module and the corresponding eigenspace is
one-dimensional. This eigenvalue $t$ should also be a root of $\widetilde
{\varphi}(\Pi)$, thus there is a vector $v\in\mathbb{R}^{d}$ with $d=\deg
(\Pi)$ such that $\mathcal{C}_{\widetilde{\varphi}(\Pi)}v=tv$. Then we get $m$
right eigenvectors of $\Phi$ linearly independent on $\mathbb{R}^{dm}$%
\[
\left(
\begin{array}
[c]{c}%
\boldsymbol{v}\\
\boldsymbol{0}\\
\vdots\\
\boldsymbol{0}%
\end{array}
\right)  ,\left(
\begin{array}
[c]{c}%
\boldsymbol{0}\\
\boldsymbol{v}\\
\vdots\\
\boldsymbol{0}%
\end{array}
\right)  ,\ldots,\left(
\begin{array}
[c]{c}%
\boldsymbol{0}\\
\boldsymbol{0}\\
\vdots\\
\boldsymbol{v}%
\end{array}
\right)  .
\]
Since the eigenspace considered is one-dimensional, this means that $m=1$ and
we are done.
\end{proof}

\begin{corollary}
\label{Cor_Delta}There exist $\Delta\in\mathbb{A}$ and for each irreducible
$k$-bounded partition $\kappa$ a polynomial $P_{\kappa}\in\mathbb{A}[T]$ such
that
\[
s_{\kappa}^{(k)}=\frac{1}{\Delta}P_{\kappa}(s_{(1)}).
\]
In particular, for any morphism $\varphi:\Lambda_{(k)}\rightarrow\mathbb{R}$
such that $\varphi(\Delta)\neq0$ we have%
\[
\varphi(s_{\kappa}^{(k)})=\frac{1}{\varphi(\Delta)}\varphi(P_{\kappa}%
)(\varphi(s_{(1)})).
\]

\end{corollary}

\begin{proof}
Since $s_{(1)}$ is a primitive element for $\mathbb{L}$ regarded as an
extension of $\mathbb{K}$, $\{1,s_{(1)},\ldots,s_{(1)}^{k!-1}\}$ is a
$\mathbb{K}$-basis of $\mathbb{L}$.\ It then suffices to consider the matrix
$M$ whose columns are the vectors $s_{(1)}^{i},i=0,\ldots,k!-1$ expressed on
the basis $\mathcal{I}=\{s_{\kappa}^{(k)},\kappa\in\mathcal{P}_{\mathrm{irr}%
}\}$. Its inverse can be written $M^{-1}=\frac{1}{\det M}N$ where the entries
of $N$ belongs to $\mathbb{A}$. So we have $\Delta=\det(M)$ and the entries on
each columns of the matrix $N$ give the polynomials $P_{\kappa},\kappa
\in\mathcal{P}_{\mathrm{irr}}$.
\end{proof}

\begin{remark}
In fact we get the equality of $\mathbb{A}$-modules $\Lambda_{(k)}=\frac
{1}{\Delta}\mathbb{A}[s_{1}]$.\ In particular, the polynomial $\Delta$ (once
assumed monic) only depends on $\Lambda_{(k)}$ and $\mathbb{A}[s_{1}]$ and not
on the choice of the bases considered in these $\mathbb{A}$-modules. Indeed a
basis change will multiply $\Delta$ by an invertible element in $\mathbb{A}$,
that is by a nonzero real.
\end{remark}

\begin{example}
For $k=2$ we get
\[
\boldsymbol{\Phi}=\left(
\begin{array}
[c]{cc}%
0 & s_{R_{1}}+s_{R_{2}}\\
1 & 0
\end{array}
\right)  \text{ and }M=I_{2}.
\]

\end{example}

\begin{example}
\label{Examk=3}For $k=3$ and with the same convention as in Example
\ref{ex_k3_Phi}, we get%
\begin{multline*}
M=\left(
\begin{array}
[c]{cccccc}%
1 & 0 & 0 & s_{R_{1}}+s_{R_{3}} & 2s_{R_{2}} & 0\\
0 & 1 & 0 & 0 & s_{R_{1}}+s_{R_{3}} & 4s_{R_{2}}\\
0 & 0 & 1 & 0 & 0 & s_{R_{1}}+3s_{R_{3}}\\
0 & 0 & 1 & 0 & 0 & 3s_{R_{1}}+s_{R_{3}}\\
0 & 0 & 0 & 2 & 0 & 0\\
0 & 0 & 0 & 0 & 2 & 0
\end{array}
\right)  \text{ and }\\
M^{-1}=\left(
\begin{array}
[c]{cccccc}%
1 & 0 & 0 & 0 & -\frac{s_{R_{1}}+s_{R_{3}}}{2} & -s_{R_{2}}\\
0 & 1 & \frac{2s_{R_{2}}}{s_{R_{1}}-s_{R_{3}}} & \frac{2s_{R_{2}}}{s_{R_{3}%
}-s_{R_{1}}} & 0 & -\frac{s_{R_{1}}+s_{R_{3}}}{2}\\
0 & 0 & \frac{3s_{R_{1}}+s_{R_{3}}}{2s_{R_{1}}-2s_{R_{3}}} & \frac{s_{R_{1}%
}+3s_{R_{3}}}{2s_{R_{3}}-2s_{R_{1}}} & 0 & 0\\
0 & 0 & 0 & 0 & \frac{1}{2} & 0\\
0 & 0 & 0 & 0 & 0 & \frac{1}{2}\\
0 & 0 & \frac{-1}{2s_{R_{1}}-2s_{R_{3}}} & \frac{-1}{2s_{R_{3}}-2s_{R_{1}}} &
0 & 0
\end{array}
\right)  \allowbreak.
\end{multline*}
So in particular, $s_{(2,1,1)}^{(3)}=\frac{1}{2}s_{(1)}^{4}-\frac{1}%
{2}(s_{R_{1}}+s_{R_{3}})s_{(1)}-s_{R_{2}}$.
\end{example}

\subsection{Algebraic variety associated to fixed values of rectangles}

Recall that $\Lambda_{(k)}=\mathbb{R}[h_{1},\ldots,h_{k}]$ and each rectangle
Schur polynomial can be written $s_{R_{a}}=JT_{a}(h_{1},\ldots,h_{k})$ for any
$a=1,\ldots,k$ where $JT_{a}\in\mathbb{R}[h_{1},\ldots,h_{k}]$ is given by the
Jacobi-Trudi determinantal formula. Consider $\vec{r}=(r_{1},\ldots,r_{k}%
)\in\mathbb{R}_{\geq0}^{k}$.

\begin{definition}
\label{Def_Variety}Let $\mathcal{R}_{\vec{r}}$ be the algebraic variety of
$\mathbb{R}^{k}$ defined by the equations $s_{R_{a}}=r_{a}$ for any
$a=1,\ldots,k$.
\end{definition}

We can consider the algebra $\overline{\Lambda}_{(k)}:=\Lambda_{(k)}/J$ where
$J$ is the ideal generated by the relations $s_{R_{a}}=r_{a}$ for any
$a=1,\ldots,k$. Write $\overline{\varphi}:\Lambda_{(k)}\rightarrow
\Lambda_{(k)}/J$ for the canonical projection obtained by specializing in
$\Lambda_{(k)}$ each rectangle Schur function $s_{R_{a}}$ to $r_{a}$.\ We
shall write for short $\overline{b}=\overline{\varphi}(b)$ for any
$b\in\Lambda_{(k)}$. Clearly $\overline{\Lambda}_{(k)}=\Lambda_{(k)}/J$ is a
finite-dimensional $\mathbb{R}$-algebra and $\overline{\Lambda}_{(k)}%
=\mathrm{vect}\langle\overline{s}_{\kappa}\mid\kappa$ irreducible$\rangle$.
The following proposition shows that the non-cancellation of $\Delta$ can be
naturally interpreted as a condition for the multiplication by $\overline
{s}_{1}$ to be a cyclic morphism in $\overline{\Lambda}_{(k)}$.

\begin{proposition}
\label{primitive_Specialization} \ 

\begin{enumerate}
\item The algebra $\overline{\Lambda}_{(k)}$ has dimension $k!$ over
$\mathbb{R}$ and $\{\overline{s}_{\kappa}\mid\kappa\in\mathcal{P}%
_{\mathrm{irr}}\}$ is a basis of $\overline{\Lambda}_{(k)}$.

\item We have $\overline{\Lambda}_{(k)}=\mathbb{R}[\overline{s}_{1}]$ if and
only if $\overline{\Delta}\neq0$.
\end{enumerate}
\end{proposition}

\begin{proof}
1:
For any $k$-bounded partition $\lambda$, write $\lambda=R_{1}^{m_{1}}%
\sqcup\cdots\sqcup R_{k}^{m_{k}}\sqcup\kappa(\lambda)$ for its decomposition into
rectangles and irreducible partition, and set $u(\lambda)=r_{1}^{m_{1}}\cdots
r_{k}^{m_{k}}$.\ If all $m_{i}$ are equal to zero for $1\leq i\leq k$, we set $u(\lambda)=0$. Let us prove that $J$ regarded as a $\mathbb{R}$-vector space has basis
$\{s_{\lambda}^{(k)}-u(\lambda)s_{\kappa(\lambda)}^{(k)}\mid\lambda\in \mathcal{B}_{k}\setminus \mathcal{P}_{\mathrm{irr}}\}$. First, the latter set is linearly independent and is thus a basis of a vector subspace $V\subset \Lambda_{(k)}$. It remains to prove that $V=J$. 

Let $\lambda=R_{1}^{m_{1}}%
\sqcup\cdots\sqcup R_{k}^{m_{k}}\sqcup\kappa(\lambda)$ be a reducible $k$-bounded partition and let $\kappa'$ be a $k$-bounded partition. Then,
\begin{align*}
s_{\kappa'}^{(k)} (s_{\lambda}^{(k)}-u(\lambda)s_{\kappa(\lambda)}^{(k)})=&s_{\kappa'}^{(k)}s_{\kappa(\lambda)}^{(k)}\left(s_{R_{1}^{m_{1}}%
\sqcup\cdots\sqcup R_{k}^{m_{k}}}^{(k)}-u(R_{1}^{m_{1}}%
\sqcup\cdots\sqcup R_{k}^{m_{k}})\right)\\
=&\sum_{\nu\in \mathcal{B}_{k}}c_{\kappa(\lambda),\kappa'}^\nu s_{\nu}^{(k)}\left(s_{R_{1}^{m_{1}}%
\sqcup\cdots\sqcup R_{k}^{m_{k}}}^{(k)}-u(R_{1}^{m_{1}}%
\sqcup\cdots\sqcup R_{k}^{m_{k}})\right)
\end{align*}
Let us write $\lambda'=R_{1}^{m_{1}}%
\sqcup\cdots\sqcup R_{k}^{m_{k}}$. For each $\nu\in\mathcal{B}_{k}$ written $\nu=R_{1}^{n_{1}}%
\sqcup\cdots\sqcup R_{k}^{n_{k}}\sqcup\kappa(\nu):=\nu'\sqcup \kappa(\nu)$, we have 
\begin{align*}
s_{\nu}^{(k)}(s_{\lambda'}^{(k)}-u(\lambda'))=&s_{\kappa(\nu)}^{(k)}(s_{\lambda'\sqcup \nu'}^{(k)}-u(\lambda')s_{\nu'}^{(k)})\\
=&s_{\kappa(\nu)}^{(k)}\left(s_{\lambda'\sqcup \nu'}^{(k)}-u(\lambda'\sqcup \nu')+u(\lambda'\sqcup \nu')-u(\lambda')s_{\nu'}^{(k)}\right)\\
=&s_{\kappa(\nu)}^{(k)}\left(s_{\lambda'\sqcup \nu'}^{(k)}-u(\lambda'\sqcup \nu')-u(\lambda')(s_{\nu'}^{(k)} -u(\nu'))\right)\\
=&s_{\lambda'\sqcup \nu'\sqcup \kappa(\nu)}^{(k)}-u(\lambda'\sqcup \nu'\sqcup \kappa(\nu))s_{\kappa(\nu)}^{(k)}-u(\lambda')(s_{\nu}^{(k)}-u(\nu)s_{\kappa(\nu)}^{(k)})\in V,
\end{align*}
where we have used $u(\lambda'\sqcup \nu')=u(\lambda'\sqcup \nu'\sqcup \kappa(\nu))$ on the last equality. Hence, $s_{\kappa'}^{(k)} (s_{\lambda}^{(k)}-u(\lambda)s_{\kappa(\lambda)}^{(k)})\in V$ and $V$ is an ideal of $\Lambda_{(k)}$. Since all generator $s_{R_{a}}-r_{a}$ of $J$ are in $V$, we have in particular $J\subset V$. 

Let us show that $(s_{\lambda}^{(k)}-u(\lambda)s_{\kappa(\lambda)}^{(k)})\in J$ for all reducible $k$-bounded partition $\lambda=R_{1}^{m_{1}}%
\sqcup\cdots\sqcup R_{k}^{m_{k}}\sqcup\kappa(\lambda)$. Since $(s_{\lambda}^{(k)}-u(\lambda)s_{\kappa(\lambda)}^{(k)})=s_{\kappa(\lambda)}^{(k)}(s_{R_{1}^{m_{1}}%
\sqcup\cdots\sqcup R_{k}^{m_{k}}}^{(k)}-u(R_{1}^{m_{1}}%
\sqcup\cdots\sqcup R_{k}^{m_{k}}))$, we just have to show that $s_{R_{1}^{m_{1}}%
\sqcup\cdots\sqcup R_{k}^{m_{k}}}^{(k)}-u(R_{1}^{m_{1}}%
\sqcup\cdots\sqcup R_{k}^{m_{k}})\in J$ for all tuple $\vec{m}=(m_{1},\ldots,m_{k})$ different from zero. We prove this by induction on $l(\vec{m})=\sum_{i}^k m_{i}$. For $\vec{m}$ such that $l(\vec{m})=1$, this is true. Suppose that $\vec{m}$ is such that $l(\vec{m})>1$ and assume without loss of generality that $m_{1}\geq 1$. Then,
\begin{align*}
s_{R_{1}^{m_{1}}%
\sqcup\cdots\sqcup R_{k}^{m_{k}}}^{(k)}  -u(R_{1}^{m_{1}}%
\sqcup\cdots\sqcup R_{k}^{m_{k}})=s_{R_{1}}& (s_{R_{1}^{m_{1}-1}%
\sqcup\cdots\sqcup R_{k}^{m_{k}}}^{(k)} -u(R_{1}^{m_{1}-1}%
\sqcup\cdots\sqcup R_{k}^{m_{k}}))\\
&+u(R_{1}^{m_{1}-1}%
\sqcup\cdots\sqcup R_{k}^{m_{k}})(s_{R_{1}}-u(R_{1}))\in J,
\end{align*}
where we have used the induction hypothesis to prove that the first term of the right hand side is in $J$. Hence, $V\subset J$, and finally $V=J$.

Observe also that $\{s_{\lambda}^{(k)}-u(\lambda)s_{\kappa(\lambda)}^{(k)}\mid\lambda$ $k$-bounded partition$\}$ is a basis of $\Lambda_{(k)}$, since it is a triangular change of basis from the standard basis $\{ s_{\lambda}^{k}, \lambda$ $k$-bounded partition$\}$. Hence, we have a decomposition 
$$\Lambda_{(k)}=\mathrm{Vect}\langle s_{\lambda}^{(k)}-u(\lambda)s_{\kappa(\lambda)}^{(k)}\mid\lambda\in \mathcal{P}_{\mathrm{irr}}\rangle\oplus J.$$
We deduced that a basis of $\Lambda_{(k)}/J$ is given by $\{ \overline{\varphi}(s_{\lambda}^{(k)}-u(\lambda)s_{\kappa(\lambda)}^{(k)}),\mid\lambda\in\mathcal{P}_{\mathrm{irr}}\}$. Since $u(\lambda)=0$ and $\lambda=\kappa(\lambda)$ when $\lambda$ is irreducible, we get that $\{ \overline{s}_{\kappa},\mid\kappa\in\mathcal{P}_{\mathrm{irr}}\}$ is a basis of $\Lambda_{(k)}/J$.

2: We have $\overline{\Delta}\neq0$ if and only if $\{\overline{s}_{1}^{r}%
\mid0\leq r\leq k!-1\}$ is a basis of $\overline{\Lambda}_{(k)}$ since
$\{\overline{s}_{\kappa}\mid\mathcal{P}_{\mathrm{irr}}\}$ is a basis of
$\overline{\Lambda}_{(k)}$ and $\overline{\Delta}$ is then the determinant
between the two bases.
\end{proof}

The following proposition is classical, we prove it for completion.

\begin{proposition}
\label{Prop_Variety}The algebraic variety $\mathcal{R}_{\vec{r}}$ is finite.
\end{proposition}

\begin{proof}
It suffices to see that the algebraic variety $\mathcal{R}_{\vec{r}%
}^{\mathbb{C}}$ of $\mathbb{C}^{k}$ defined by the equations $s_{R_{a}}=r_{a}$
for any $a=1,\ldots,k$ is finite.\ We can decompose $\mathcal{R}_{\vec{r}%
}^{\mathbb{C}}=V_{1}\cup\cdots\cup V_{m}$ into its irreducible components.\ To
each such component $V_{j}$ is associated a prime ideal $J_{j}$ and we have
$J=J_{1}\cap\cdots\cap J_{m}$. Therefore for any $j=1,\ldots,m,$
$\mathbb{C}[h_{1},\ldots,h_{k}]/J_{j}$ is a finite-dimensional algebra
(because $\mathbb{C}[h_{1},\ldots,h_{k}]/J$ is) which is an integral domain.
So $\mathbb{C}[h_{1},\ldots,h_{k}]/J_{j}$ is in fact a field and $J_{j}$ is
maximal in $\mathbb{C}[h_{1},\ldots,h_{k}]$. By using Hilbert's
Nullstellensatz's theorem, we obtain that each $V_{j}$ reduces to a point, so
$\mathcal{R}_{\vec{r}}^{\mathbb{C}}$ is finite.
\end{proof}

\section{Nonnegative morphisms on $\Lambda_{(k)}$}

\subsection{Nonnegative morphisms with $\Phi$ irreducible}

We show now that when the matrix $\Phi$ introduced in
\S \ \ref{subsec_MatrixFI} is irreducible, the values of $\varphi$ on the
rectangle Schur functions $s_{R_{a}},$ $1\leq a\leq k$ determine completely
the morphism $\varphi$. Recall we have denoted by $\mathbb{A}=\mathbb{R}%
[s_{R_{1}},\ldots,s_{R_{k}}]$ the subalgebra of $\Lambda_{(k)}=\mathbb{R}%
[h_{1},\ldots,h_{k}]$ generated by the $k$-rectangle Schur functions. Also we
have $\mathcal{I}=\{s_{\kappa}^{(k)}\mid\kappa\in\mathcal{P}_{\mathrm{irr}}%
\}$. Set $\mathcal{R}=\{s_{R_{a}}\}_{1\leq a\leq k}$.

\begin{theorem}
\label{principalDomain} \ 

\begin{enumerate}
\item Let $\varphi:\mathbb{A}\rightarrow\mathbb{R}$ be a morphism, nonnegative
on $\mathcal{R}$ and such that its associated matrix $\Phi$ is irreducible.
Then there exists a unique morphism $\tilde{\varphi}:\Lambda_{(k)}%
\longrightarrow\mathbb{R}_{\geq0}$ extending $\varphi$, nonnegative on the
$k$-Schur functions and positive on $\mathcal{I}$.

\item A positive morphism $\varphi:\Lambda_{(k)}\longrightarrow\mathbb{R}$ is
uniquely determined by its values on $\mathcal{R}$.
\end{enumerate}
\end{theorem}

\begin{proof}
(1): Let us prove the existence of $\tilde{\varphi}$. Set $(r_{1},\ldots,
r_{k})=(\varphi(s_{R_{1}}),\ldots,\varphi(s_{R_{k}}))$, and assume first that
$\Delta(r_{1},\ldots,r_{k})\not =0$ where $\Delta$ is the polynomial defined
in Corollary \ref{Cor_Delta}. We have to show that there exists a morphism
$\widetilde{\varphi}$ on $\Lambda_{(k)}$ such that $\widetilde{\varphi}$ is
positive on $\mathcal{I}$ and $\widetilde{\varphi}(s_{R_{a}})=r_{a}$ for
$a=1,\ldots,k$. The set of morphisms from $\Lambda_{(k)}$ to $\mathbb{R}$
which takes values $r_{a}$ on $s_{R_{a}}$ for $a=1\ldots k$ is in bijection
with the set of morphisms from $\overline{\Lambda}_{(k)}$ to $\mathbb{R}.$
Here recall that $\overline{\Lambda}_{(k)}=\Lambda_{(k)}/J$ with $J$ the ideal
generated by the relations $s_{R_{a}}=r_{a}$ for $a=1,\ldots,k$. Since we have
assumed $\bar{\Delta}\not =0$, Proposition \ref{primitive_Specialization}
yields that $\overline{\Lambda}_{(k)}=\mathbb{R}[\overline{s}_{1}]$. There
exists one morphism from $\mathbb{R}[\overline{s}_{1}]$ to $\mathbb{R}$ for
each real root of the minimal polynomial of $\overline{s}_{1}$, which is
$\overline{\Xi}$ because $\deg(\overline{\Xi})=k!=\dim(\overline{\Lambda
}_{(k)})$. Let $t$ be the root of $\overline{\Xi}$ with maximal modulus. It is
positive since $\overline{\Xi}$ is the characteristic polynomial of the
irreducible matrix $\Phi_{(r_{1},\ldots,r_{k})}$, and moreover it is the
Perron-Frobenius eigenvalue of $\Phi$. Then, the specialization $\overline
{s}_{1}=t$ yields a morphism from $\overline{\Lambda}_{(k)}$ to $\mathbb{R}$,
and by extension a morphism $\widetilde{\varphi}$ from $\Lambda_{(k)}$ to
$\mathbb{R}$. In particular, the equality (\ref{PieriPhi}) holds for
$\widetilde{\varphi}$, and the vector $X=(\widetilde{\varphi}(s_{\lambda
}^{(k)}))_{\lambda\in\mathcal{P}_{\mathrm{irr}}}$ is the eigenvector of $\Phi$
corresponding to the Perron Frobenius eigenvalue $t$ and such that
$X(\emptyset)=1$. Therefore, $\widetilde{\varphi}$ is positive on
$\mathcal{I}$. So we have proved the existence of an extension of $\varphi$
with the right properties in the case where $\Delta\neq0$.

We now drop the hypothesis $\Delta=0$. Consider $\vec{r}=(r_{1},\ldots,r_{k})$
such that the matrix $\Phi$ is irreducible. For any extension $\widetilde
{\varphi}$ (if any) of $\varphi$ positive on $\mathcal{I}$, the equality
(\ref{PieriPhi}) still holds. Therefore, $X^{\vec{r}}=(\widetilde{\varphi
}(s_{\lambda}^{(k)})_{\lambda\in\mathcal{P}_{\mathrm{irr}}})$ should then be
the eigenvector of $\Phi$ corresponding to the Perron Frobenius eigenvalue
$t_{\vec{r}}$ and such that $X^{\vec{r}}(\emptyset)=1$. For $\mu\in
\mathcal{B}_{k}$, set
\begin{equation}
\label{expression_morphism}\widetilde{\varphi}_{\vec{r}}(s_{\mu}^{(k)}%
)=\prod_{a=1}^{k}r_{a}^{p_{a}^{\mu}}X^{\vec{r}}(\widetilde{\mu})\text{ with
}s_{\mu}^{(k)}=\prod_{a=1}^{k}s_{R_{a}}^{p_{a}^{\mu}}s_{\widetilde{\mu}}%
^{(k)}\text{ , }\tilde{\mu}\text{ irreducible.}%
\end{equation}
Then, the map $\widetilde{\varphi}_{\vec{r}}$ is an extension of $\varphi$
which is positive on $\mathcal{I}$ and nonnegative on the $k$-Schur functions
by construction.\ So it just remains to prove that $\widetilde{\varphi}%
_{\vec{r}}$ is a morphism.

Since $t_{\vec{r}}$ and $X^{\vec{r}}$ are continuous functions of $\vec{r}$ on
the set of irreducible matrices, the map $\widetilde{\varphi}_{\vec{r}}$ is a
continuous function of $\vec{r}$. The hypersurface $V(\Delta):=\{\Delta
(\vec{r})=0\}$ is closed in the Zariski topology, thus $V(\Delta)$ has empty
interior in the set $\Theta=\{\vec{r}\in\mathbb{R}_{\geq0}^{k}\mid\Phi
_{(r_{1},\ldots,r_{k})}$ is irreducible$\}$. By the previous arguments, for
all $\vec{r}\in\Theta$ outside $V(\Delta)$, the map $\widetilde{\varphi}%
_{\vec{r}}$ is a morphism and $\vec{r}\mapsto\widetilde{\varphi}_{\vec{r}}$ is
continuous on $\Theta$, thus $\widetilde{\varphi}_{\vec{r}}$ is a morphism for
$\vec{r}\in\overline{\Theta\setminus V(\Delta)}$. By Proposition
\ref{irredPhi}, the set $\Theta$ is open.\ Let $\vec{r}\in\Theta\cap
V(\Delta)$.\ Since the interior of $V(\Delta)$ is empty, one can define a
sequence $\vec{r}^{(n)}\in\Theta\setminus V(\Delta)$ which tends to $\vec
{r}\in\overline{\Theta\setminus V(\Delta)}$ as $n$ goes to infinity. Finally
we get that $\widetilde{\varphi}_{\vec{r}}$ is a morphism, as desired.

We now prove the uniqueness of the extension $\widetilde{\varphi}$. Let
$\widehat{\varphi}$ be another real extension of $\varphi$ positive on
$\mathcal{I}$. By (\ref{PieriPhi}), the vector $\widehat{X}=(\widehat{\varphi
}(s_{\lambda}^{(k)}))_{\lambda\in\mathcal{P}_{\mathrm{irr}}}$ is also a left
eigenvector of $\Phi$ for the positive eigenvalue $\widehat{\varphi}(s_{1})$
normalized so that $\widehat{X}_{\emptyset}$ is equal to one. Since
$\widehat{X}$ has positive entries, it is the Perron Frobenius eigenvector of
$\Phi$ and is thus equal to the vector $X^{\vec{r}}$ defined above. By the
morphism property of $\widehat{\varphi}$, for $\mu\in\mathcal{B}_{k}$ such
that $s_{\mu}^{(k)}=\prod_{a=1}^{k}s_{R_{a}}^{p_{a}^{\mu}}s_{\widetilde{\mu}%
}^{(k)}$ with $\tilde{\mu}$ irreducible, we have
\begin{equation}
\widehat{\varphi}(s_{\mu}^{(k)})=\prod_{a=1}^{k}\varphi(s_{R_{a}}^{p_{a}^{\mu
}})\widehat{X}(\widetilde{\mu})=\prod_{a=1}^{k}r_{a}^{p_{a}^{\mu}}%
X(\widetilde{\mu})=\widetilde{\varphi}(s_{\mu}^{(k)}),
\end{equation}
where we have used the definition of $\widetilde{\varphi}$ given in
(\ref{expression_morphism}). Therefore, $\widetilde{\varphi}=\widehat{\varphi
}$, and we have proven the uniqueness of the extension of $\varphi$ satisfying
the properties of the statement.

(2): Let $\varphi$ be a positive morphism on $\Lambda_{(k)}$. Then, the
associated matrix $\Phi$ is irreducible by Proposition \ref{irredPhi}. Hence,
the first part of the proposition yields that $\varphi$ is uniquely determined
by its values on $\mathcal{R}$.
\end{proof}

An immediate consequence of the latter theorem is the description of positive
extremal harmonic functions. We define an action of $\mathbb{R}_{>0}$ on
$\mathcal{F}(\mathcal{B}_{k},\mathbb{R}_{\geq0})$ by
\[
t\cdot\varphi(s_{\lambda}^{(k)})=t^{\left\vert \lambda\right\vert }%
\varphi(s_{\lambda}^{(k)}),
\]
for $t>0$, $\varphi\in\mathcal{F}(\mathcal{B}_{k},\mathbb{R}_{\geq0})$.

\begin{corollary}
\ 

\begin{enumerate}
\item Let $\varphi\in\partial\mathcal{H}^{+}(\mathcal{B}_{k})$, and suppose
that $\varphi$ is positive on the $k$-Schur functions. Then, $\varphi$ is
uniquely determined by its values on the $s_{R_{a}}$, $1\leq a\leq k$.

\item Assume the matrix $\Phi$ associated to $\varphi$ is irreducible. Then
there exists a unique $t>0$ such that $t^{-1}\cdot\varphi$ can be extended to
an element of $\partial\mathcal{H}^{+}(\mathcal{B}_{k})$ positive on
$\mathcal{I}$.
\end{enumerate}
\end{corollary}

\begin{proof}
The only non immediate statement is the second one. Suppose that $\Phi$ is
irreducible. Then, by Theorem \ref{principalDomain}, $\varphi$ can be extended
in a unique way to a morphism $\widetilde{\varphi}$ nonnegative on
$\mathcal{B}_{k}$ and positive on $\mathcal{I}$. Let $t=\widetilde{\varphi
}(s_{1})>0$. Then, $t^{-1}\cdot\widetilde{\varphi}$ is a nonnegative morphism
on $\mathcal{B}_{k}$ such that $t^{-1}\cdot\widetilde{\varphi}(s_{1})=1$,
which belongs to $\partial\mathcal{H}^{+}(\mathcal{B}_{k})$. It is clear that
$t^{-1}\cdot\widetilde{\varphi}$ extends $t^{-1}\cdot\varphi$. Also if
$\widehat{\theta}\in\partial\mathcal{H}^{+}(\mathcal{B}_{k})$ extends
$s^{-1}\cdot\varphi$ with $s>0$ and is positive on $\mathcal{I}$, then the
vector with entries $s\cdot\widehat{\theta}(s_{\lambda}^{(k)}),\lambda
\in\mathcal{P}_{\mathrm{irr}}$ is an eigenvector for $\Phi$ associated to the
eigenvalue $s$. Since it has positive entries, we get $t=s$ and $\widehat
{\theta}=t^{-1}\cdot\widetilde{\varphi}$.
\end{proof}

\begin{remark}
\label{Rem_param}It follows also from Proposition \ref{irredPhi} and Assertion
2 of Theorem \ref{principalDomain}, that each morphism $\varphi:\mathbb{A}%
\rightarrow\mathbb{R}$ positive on $\mathcal{R}$ can be extended in a unique
way to a morphism $\widetilde{\varphi}:\Lambda_{(k)}\longrightarrow
\mathbb{R}_{\geq0}$ positive on the $k$-Schur functions. Also clearly, two
distinct such morphisms on $\mathbb{A}$ will yield distinct extensions on
$\Lambda_{(k)}$. Thus the morphisms $\widetilde{\varphi}:\Lambda
_{(k)}\longrightarrow\mathbb{R}_{\geq0}$ positive on the $k$-Schur functions
are parametrized by $\mathbb{R}_{>0}^{k}$ via the map which associates to each
such morphism its values on $\mathcal{R}$.
\end{remark}

\subsection{Two parametrizations of the positive morphisms}

\label{SubSec_Twoparam}The more immediate parametrization of the positive
morphisms $\varphi:\Lambda_{(k)}\rightarrow\mathbb{R}$ such that
$\varphi(s_{\lambda}^{(k)})>0$ for any $k$-bounded partition is obtained from
the factorization property (Corollary \ref{reducedCase}) of the $k$-Schur
functions. Consider
\[
V=\left\{  (h_{1},\ldots,h_{k})\in\mathbb{R}^{k}\mid\left\{
\begin{array}
[c]{c}%
JT_{R_{i}}(h_{1},\ldots,h_{k})>0,i=1,\ldots,k\\
JT_{\kappa}(h_{1},\ldots,h_{k})>0\ \forall\kappa\in\mathcal{P}_{\mathrm{irr}}%
\end{array}
\right.  \right\}  \subset\mathbb{R}_{>0}^{k}%
\]
where we have set $s_{Ri}=JT_{R_{i}}(h_{1},\ldots,h_{k})$ and $s_{\kappa
}=JT_{\kappa}(h_{1},\ldots,h_{k})$ where $JT_{R_{1}},\ldots,JT_{R_{k}}$ and
$JT_{\kappa},\kappa\in\mathcal{P}_{\mathrm{irr}}$ are polynomials in
$\mathbb{R}[X_{1},\ldots,X_{k}]$. To each point in $V$ corresponds a unique
positive morphism $\varphi$ defined on $\Lambda_{(k)}$. Now define%
\[
U=\{\vec{r}=(r_{1},\ldots,r_{k})\in\mathbb{R}_{>0}^{k}\}.
\]
As explained in Remark \ref{Rem_param}, the positive morphisms $\varphi
:\Lambda_{(k)}\rightarrow\mathbb{R}$ such that $\varphi(s_{\lambda}^{(k)})>0$
for any $\lambda\in\mathcal{B}_{k}$ are parametrized by $U$ and we can define
a map $f:U\rightarrow V$ such that
\[
f(r_{1},\ldots,r_{k})=(\varphi(h_{1}),\ldots,\varphi(h_{k})).
\]
The map $f$ is then continuous on $U$ since the entries of the matrix $\Phi$
are and so is its Perron Frobenius vector normalized at $1$ on $s_{\emptyset}%
$. Moreover, the map $f$ is bijective by Theorem \ref{principalDomain} and we
have
\[
f^{-1}:\left\{
\begin{array}
[c]{c}%
V\rightarrow U\\
(h_{1},\ldots,h_{k})\mapsto(JT_{R_{1}}(h_{1},\ldots,h_{k}),\ldots,JT_{R_{k}%
}(h_{1},\ldots,h_{k}))
\end{array}
\right.
\]
where the polynomials $JT_{R_{1}},\ldots,JT_{R_{k}}$ are given by the
Jacobi-Trudi determinantal formulas. In particular $f^{-1}$ is continuous on
$V$.

\begin{lemma}
\label{Lemma_Bounded}The map $f$ is bounded on any bounded subset of $U$.
\end{lemma}

\begin{proof}
Let $B\subset U$ be a bounded subset of $U$. By definition of $f$, for any
$\vec{r}=(r_{1},\ldots,r_{k})$ in $B$, $\varphi(h_{1})$ is the first
coordinate of $f(r_{1},\ldots,r_{k})$ and coincides with the Perron Frobenius
eigenvalue of the matrix $\Phi$, that is with its spectral radius.\ Since the
spectral radius of a real matrix is a bounded function of its entries, we get
that $\varphi(h_{1})$ is bounded when $\vec{r}$ runs over $B$.\ To conclude,
observe that for any $a=2,\ldots,k$ we have $\varphi(h_{a})\leq\varphi
(h_{1})^{a}$ because $h_{1}=s_{1},$ the map $\varphi$ is multiplicative and
$h_{a}$ appears in the decomposition of $s_{1}^{a}$ on the basis of $k$-Schur
functions (which only makes appear nonnegative real coefficients).
\end{proof}

Now set
\begin{align*}
\overline{U}  &  =\{\vec{r}=(r_{1},\ldots,r_{k})\in\mathbb{R}_{\geq0}%
^{k}\}\text{ and }\\
\overline{V}  &  =\left\{  (h_{1},\ldots,h_{k})\in\mathbb{R}^{k}\mid\left\{
\begin{array}
[c]{c}%
JT_{R_{i}}(h_{1},\ldots,h_{k})\geq0,i=1,\ldots,k\\
JT_{\kappa}(h_{1},\ldots,h_{k})\geq0\ \forall\kappa\text{ }%
k\text{-irreducible}%
\end{array}
\right.  \right\}  \subset\mathbb{R}_{\geq0}^{k}%
\end{align*}
Since $JT_{R_{1}},\ldots,JT_{R_{k}}$ are polynomials, we can extend $f^{-1}$
by continuity on $\overline{V}$ and get a continuous map $g:\overline
{V}\rightarrow\overline{U}$. But this is not immediate right now that $g$ is
bijective and $f$ can also be extended to a bijective map from $\overline{U}$
to $\overline{V}$. Observe nevertheless that if we can extend $f$ by
continuity on $\overline{U}$, the continuity of $g$ and $f$ will imply that
$f\circ g=id_{\overline{V}}$ and $g\circ f=id_{\overline{U}}.$ Therefore, to
extend $f$ by continuity will suffice to prove that $\overline{U}\ $and
$\overline{V}$ are homeomorphic by $f$.

\subsection{Extension of the map $f$ on $\overline{U}$}

Let $\vec{r}_{0}\in\mathbb{R}_{\geq0}^{k}$, and denote by $A(\vec{r}_{0})$ the
set of limiting values of $f(\vec{r})$ as $\vec{r}$ goes to $\vec{r}_{0}$ in $\overline{U}=\mathbb{R}^{k}$.
Recall the notation of the previous paragraph, in particular the function $g$
is defined and continuous on $\overline{V}$ and $g=f^{-1}$ on $f(U)$.

\begin{lemma}
The set $A(\vec{r}_{0})$ is a connected subset of $\mathcal{R}_{\vec{r}_{0}}$
(see Definition \ref{Def_Variety}).
\end{lemma}

\begin{proof}
Consider $K_{n}=B\left(  \vec{r}_{0},\frac{1}{n}\right)  \cap\mathbb{R}%
_{>0}^{k}$.\ This is a system of decreasing bounded connected neighborhoods of
$\vec{r}_{0}$ in $\mathbb{R}_{>0}^{k}$. By definition, $A(\vec{r}_{0}%
)=\bigcap_{n\geq1}\overline{f(K_{n})}$. By Lemma \ref{Lemma_Bounded}, we know
that $f$ is bounded on bounded subsets of $U=\mathbb{R}_{>0}^{k}$, therefore
we get that $f(K_{n})$ is bounded and thus $\overline{f(K_{n})}$ is compact.
Since $f$ is continuous on $U$ and $K_{n}$ is connected, $f(K_{n})$ is also
connected, which implies that $\overline{f(K_{n})}$ is connected. Hence,
$A(\vec{r}_{0})$ is a decreasing intersection of connected compact sets, and
thus $A(\vec{r}_{0})$ is connected.

Let $\vec{h}\in A(\vec{r}_{0})$. Then, there exists a sequence $(\vec
{r_{n}})_{n\geq1}$ in $U$ converging to $\vec{r}_{0}$ such that $\vec{h}%
_{n}:=f(\vec{r_{n}})$ converges to $\vec{h}$ as $n$ goes to infinity. Since $g=f^{-1}$ on $\mathbb{R}_{>0}^{k}$, $g(\vec{h}_{n})=g\circ f(\vec
{r_{n}})=\vec{r_{n}}$ for $n\geq1$. Moreover, since $g$ is continuous and
$(\vec{h}_{n})_{n\geq1}$ converges to $\vec{h}$ as $n$ goes to infinity,
\[
g(\vec{h})=\lim_{n\rightarrow\infty}g(\vec{h}_{n})=\lim_{n\rightarrow\infty
}\vec{r_{n}}=\vec{r}_{0}%
\]
which implies that $\vec{h}\in\mathcal{R}_{\vec{r}_{0}}$.
\end{proof}

\begin{theorem}
\ \label{Th_UVhomeo}

\begin{enumerate}
\item The map $f$ is an homeomorphism from $\overline{U}$ to $\overline{V}$.

\item The morphisms $\varphi:\Lambda_{(k)}\rightarrow\mathbb{R}$ nonnegative
on the $k$-Schur functions are parametrized by $\mathbb{R}_{\geq0}^{k}$.
\end{enumerate}
\end{theorem}

\begin{proof}
(1): The set $A(\vec{r}_{0})$ is a connected subset by the previous lemma and
it is also finite by Proposition \ref{Prop_Variety}. Therefore, the set
$A(\vec{r}_{0})$ is a singleton. In particular, $f(\vec{r})$ converges to some
$f(\vec{r}_{0})$ as $\vec{r}$ tends to $\vec{r}_{0}$, and $f$ can be extended
by continuity on $\mathbb{R}_{\geq0}^{k}$. As explained at the end of
\S \ \ref{SubSec_Twoparam}, this suffices to conclude that $f$ is an
homeomorphism from $\overline{U}$ to $\overline{V}$.

(2): By the first part of the theorem, it suffices to prove that any morphism
$\varphi:\Lambda_{(k)}\rightarrow\mathbb{R}$ nonnegative on the $k$-Schur
functions is in the closure of the set of positive morphisms. Let $\varphi
_{0}:\Lambda_{(k)}\rightarrow\mathbb{R}$ be a positive morphism (for example,
one can take the image of any element of $U$ by $f$). Then, for any
nonnegative morphism $\varphi:\Lambda_{(k)}\rightarrow\mathbb{R}$ and $0\leq
t\leq1$, one can define the convolution morphism $\varphi*_{t}\varphi_{0}$ by
the formula
\[
\varphi*_{t}\varphi_{0}(f)=((1-t)\cdot\varphi\otimes t\cdot\varphi_{0}%
)\Delta(f)
\]
for $f\in\Lambda_{(k)}$, where $\Delta$ is the coproduct on $\Lambda_{(k)}$.
Since $\Delta$ is an algebra morphism from $\Lambda_{(k)}$ to $\Lambda
_{(k)}\otimes\Lambda_{(k)}$, $\varphi*_{t}\varphi_{0}$ is indeed a morphism.
Let $\lambda\in\mathcal{B}_{k}$ and $s_{\lambda}^{(k)}$ the corresponding
$k$-Schur function. Then, by \cite[Corollary 8.1]{LamLR},
\[
\Delta(s_{\lambda}^{(k)})=\sum_{\substack{\mu,\nu\in\mathcal{B}_{k}\\\vert
\mu\vert+\vert\nu\vert=\vert\lambda\vert}}C_{\mu,\nu}^{\lambda,(k)}s_{\mu
}^{(k)}\otimes s_{\nu}^{(k)}%
\]
with nonnegative coefficients $C_{\mu,\nu}^{\lambda,(k)}$. Since $\Delta$ is
the usual coproduct from the ring of symmetric functions, we moreover have
$C_{\lambda,\emptyset}^{\lambda,(k)}=C_{\emptyset,\lambda}^{\lambda,(k)}=1$,
so that
\[
\varphi*_{t}\varphi_{0}(s_{\lambda}^{(k)})=(1-t)\cdot\varphi(s_{\lambda}%
^{(k)})+t\cdot\varphi_{0}(s_{\lambda}^{(k)})+\sum_{\substack{\vert\mu
\vert+\vert\nu\vert=\vert\lambda\vert\\\mu,\nu\not =\emptyset}}C_{\mu,\nu
}^{\lambda,(k)}(1-t)^{\vert\mu\vert}\varphi(s_{\mu}^{(k)})t^{\vert\nu\vert}
\varphi_{0}(s_{\nu}^{(k)}).
\]
Since $\varphi_{0}(s_{\lambda}^{(k)})$ is positive and all terms in the above
sums are nonnegative, $\varphi*_{t}\varphi_{0}(s_{\lambda}^{(k)})$ is positive
for all $t>0$. Hence, $\varphi*_{t}\varphi_{0}$ is a positive morphism and
$\varphi*_{t}\varphi_{0}$ converges to $\varphi$ as $t$ goes to zero. Thus,
$\varphi$ is in the closure of the set of positive morphisms.
\end{proof}

\begin{example}
\label{Ex_k=2}For $k=2$, we get for the matrix associated to $\vec{r}%
=(r_{1},r_{2})\in\mathbb{R}_{\geq0}^{2}$%
\[
\Phi=\left(
\begin{array}
[c]{cc}%
0 & r_{1}+r_{2}\\
1 & 0
\end{array}
\right)
\]
whose greatest eigenvalue is $\sqrt{r_{1}+r_{2}}$ with associated normalized
left eigenvector $(1,\sqrt{r_{1}+r_{2}})$. We thus get $\vec{h}=f(\vec
{r})=(\sqrt{r_{1}+r_{2}},r_{1})$ since $h_{1}=\sqrt{r_{1}+r_{2}}$ and
$h_{2}=r_{1}$. Conversely, we have $g(\vec{h})=(h_{2},h_{1}^{2}-h_{2})$. If we
assume $h_{1}=1$, we get $\partial\mathcal{H}^{+}(\mathcal{B}_{2}%
)=\{(1,h_{2})\mid h_{2}\in\lbrack0,1]\}.$
\end{example}

\section{Markov chains on alcoves}

\subsection{Central Markov chains on alcoves from harmonic functions}

\label{definitionCentral_Markov}

Recall the notation of \S \ \ref{subset_Lattices} for the notion of reduced
alcove paths.\ A probability distribution on reduced alcove paths is said
central when the probability $p_{\pi}$ of the path $\pi=(A_{1}=A^{(0)}%
,A_{2},\ldots,A_{m})$ only depends on $m,A_{1}$ and $A_{m}$, that is only on
its length and its alcoves ends. In the situation we consider, affine
Grassmannian central random paths correspond to central random paths on
$\mathcal{B}_{k}$. Similarly, affine (non Grassmannian) central random alcove
paths correspond to central random paths on the Hasse diagram $\mathcal{G}%
_{k}$ of the weak Bruhat order. They are determined by the positive harmonic
functions on $\mathcal{B}_{k}$ and $\mathcal{G}_{k}$,
respectively$\mathcal{\ }$(see \cite{Ker}).

More precisely any central probability distribution on the affine Grassmannian
alcove paths can be written
\[
p_{\pi}=\frac{h(\mu)}{h(\lambda)}%
\]
where $h\in\mathcal{H}^{+}(\mathcal{B}_{k})$ is positive and for any path
$\pi=(A_{1},\ldots,A_{m}),$ $\mu$ and $\lambda$ are the $k$-bounded partitions
associated to $A_{1}$ and $A_{m}$. Also we then get a Markov chain on
$\mathcal{B}_{k}$ (or equivalently on the affine Grassmannian elements) with
transition matrix
\[
\Pi(\lambda,\mu)=\frac{h(\mu)}{h(\lambda)}.
\]
When $h$ is extremal, it corresponds to a morphism $\varphi$ on $\Lambda
_{(k)}$ with $\varphi(s_{(1)})=1$, nonnegative on the $k$-Schur functions. We
get an extremal central distribution on the trajectories starting at $A^{(0)}$
verifying $p_{\pi}=\frac{\varphi(s_{\mu}^{(k)})}{\varphi(s_{\lambda}^{(k)})}%
$.\ The associated Markov chain has then the transition matrix $\Pi
(\lambda,\mu)=\frac{\varphi(s_{\mu}^{(k)})}{\varphi(s_{\lambda}^{(k)})}$.

\subsection{Comparison with Lam's uniform distribution}
\label{SubSec_Compar}

The probability distribution on reduced alcoves paths used by Lam in
\cite{Lam2} does not coincide with ours in general: at each step of such a
path, an alcove is chosen uniformly with the condition that each hyperplane
can be crossed only once. It is not difficult to check that such a
distribution is not central when $k\geq3$. One can for example compare the
probability of two paths from the fundamental alcove to a suitable alcove on the
border of the dominant Weyl chamber with one path remaining always on the
border and not the other (or equivalently by comparing the probabilities of
two paths in the graph $\mathcal{B}_{k}$ from $\emptyset$ to a suitable column
partition with one path containing only column partitions and not the other).

The case $k=2$ is very particular because $\mathcal{B}_{2}$ has then a very
simple regular structure which imposes that the probability of any path from
$\emptyset$ to the $2$-restricted partition $\lambda$ is equal to $\left(
\frac{1}{2}\right)  ^{\left\lfloor \frac{\left\vert \lambda\right\vert }%
	{2}\right\rfloor }$. In particular, Lam's uniform distribution is then
central.\ One can check that in our setting, this correspond to the case where
$s_{R_{1}}=s_{(2)}$ and $s_{R_{2}}=s_{(1,1)}$ are both specialized to
$\frac{1}{2}$.

\subsection{Involutions on the reduced walk}

\label{SubSec_Invol}

By Corollary \ref{reducedCase}, the structure of the graph $\mathcal{B}_{k}$
is completely determined by the matrix $\mathbf{\Phi}$ depicted in Section
\ref{subsec_MatrixFI} with entries in $\mathbb{R}[s_{R_{1}},\ldots,s_{R_{k}}%
]$. Then $\Phi_{(r_{1},\ldots,r_{k})}$ is the matrix $\boldsymbol{\Phi}$ after
the specialization $s_{R_{1}}=r_{1},\ldots,s_{R_{k}}=r_{k}$. We are going to
see that this matrix exhibits particular symmetries coming from the underlying
alcove structure.

The first symmetry is due to the action of $\omega_{k}$ on $\Lambda_{(k)}$
which sends $s_{R_{a}}$ to $s_{R_{k-a}}$ for any $a=1,\ldots,k$. Since
$\omega_{k}$ is an algebra morphism, we get for $1\leq a_{i}\leq k$ and
$s\geq1$%

\[
\left\{
\begin{matrix}
\Phi_{(r_{1},\ldots,r_{k})}(\lambda,\mu)=1\Leftrightarrow\Phi_{(r_{1}%
,\ldots,r_{k})}(\lambda^{\omega_{k}},\mu^{\omega_{k}})=1,\\
\Phi_{(r_{1},\ldots,r_{k})}(\lambda,\mu)=r_{a_{1}}+\cdots+r_{a_{s}%
}\Leftrightarrow\Phi_{(r_{1},\ldots,r_{k})}(\lambda^{\omega_{k}},\mu
^{\omega_{k}})=r_{k+1-a_{1}}+\cdots+r_{k+1-a_{s}}.
\end{matrix}
\right.
\]
Hence, if we denote by $\Omega$ the matrix of the conjugation $\omega_{k}$ on
the basis of irreducible partitions, we get
\begin{equation}
\Omega\Phi_{(r_{1},\ldots,r_{k})}\Omega^{-1}=\Omega\Phi_{(r_{1},\ldots,r_{k}%
)}\Omega=\Phi_{(r_{k},\ldots,r_{1})}. \label{PhiConjugation}%
\end{equation}
For the second symmetry, we need some basic facts about the affine Coxeter
arrangement of type $A_{k}^{(1)}$. For any root $\alpha$ and any integer, let
$H_{\alpha,r}$ be the affine hyperplane
\[
H_{\alpha,r}=\{v\in\mathbb{R}^{k},\langle v,\alpha\rangle=r\}.
\]
We denote by $s_{\alpha,r}$ the reflection with respect to this hyperplane and
for $\beta$ in the weight lattice $P$, we write $t_{\beta}$ for the
translation by $\beta$. We have then $s_{\alpha,r}=t_{r\alpha}s_{\alpha,0}$.
For $w\in W$, we have the commutation relations
\begin{equation}
wt_{\beta}=t_{w(\beta)}w\,,\quad ws_{\alpha,r}=s_{w(\alpha),r}%
w\quad\text{ and }\quad t_{\beta}s_{\alpha,r}=s_{\alpha,r+\langle\beta
,\alpha\rangle}t_{\beta}. \label{commutRelation}%
\end{equation}
Affine Grassmannian elements are in bijection with alcoves in the dominant
Weyl chamber through a map $w\mapsto A_{w}$ such that $w\rightarrow w^{\prime
}$ (that is we have a covering relation for the weak order from $w$ to
$w^{\prime}$) if and only if there is a hyperplane $H_{\alpha,r}$ such that
$A_{w^{\prime}}=s_{\alpha,r}(A_{w})$. In this case, we write
$w\xrightarrow{\alpha,r}w^{\prime}$.

Write $v_{w}$ for the center of the alcove $A_{w}$ (defined as the mean of the
its extreme weights). With these notations, $w\xrightarrow{\alpha,r}w^{\prime
}$ if and only if $v_{w^{\prime}}=s_{\alpha,r}(v_{w})$ and $r<\langle
\alpha,v_{w^{\prime}}\rangle<r+1$.

Any alcove $A_{w}$ is completely determined by its center $v_{w}$.\ Let
$\mathsf{B}$ be the set of alcoves corresponding to affine Grassmannian
elements $w$ such that $\langle v_{w},\alpha_{i}\rangle\in]0,1[$ for any
$i=1,\ldots,k$ (i.e. such that the coordinates of $v_{w}$ on the basis of
fundamental weights belong to $]0,1[$). Recall also there is an involution on
the Dynkin diagram of affine type $A_{k}^{(1)}$ fixing the node $0$ and
sending each node $i\in\{1,\ldots,k\}$ to $i^{\ast}=k+1-i$. Consider now the
involution $I:\mathbb{R}^{k}\rightarrow\mathbb{R}^{k}$ defined by
\[
I=t_{\rho}\circ w_{0},
\]
where $\rho=\sum_{i=1}^{k}\Lambda_{i}$ and $w_{0}$ is the longest element of
$W.$ Observe we indeed get an involution because $w_{0}\circ t_{\rho}\circ
w_{0}=t_{-\rho}$.

\begin{lemma}
\label{Lem_IonB}The involution $I$ restricts to an involution on the set
$\mathsf{B}$ .
\end{lemma}

\begin{proof}
Consider $A\in\mathsf{B}$ with center $v$.\ Set $v=\sum_{i=1}^{k}\langle
v_{A},\alpha_{i}\rangle\Lambda_{i}$. Since $w_{0}(\Lambda_{i})=-\Lambda
_{i^{\ast}}$ we get $w_{0}(v)=\sum_{i=1}^{k}-\langle v,\alpha_{i}%
\rangle\Lambda_{i^{\ast}}$.\ We also have $\rho=\sum_{i=1}^{k}\Lambda_{i}$
which gives
\[
I(v)=t_{\rho}\circ w_{0}(v)=\sum_{i=1}^{k}(1-\langle v,\alpha_{i}%
\rangle)\Lambda_{i^{\ast}}.
\]
By hypothesis, $\langle v,\alpha_{i}\rangle\in]0,1[$ for any $i=1,\ldots k$
and thus $1-\langle v,\alpha_{i}\rangle\in]0,1[$. Now the coordinates of
$I(v)$ on the basis of fundamental weights all belong to $]0,1[$. This implies
that $I(v)$ is the center of an alcove in $\mathsf{B}$ and $I$ restricts to an
involution on $\mathsf{B}$.
\end{proof}

\begin{lemma}
\label{stabilize_without_coeff} If $A,A^{\prime}$ are two alcoves of
$\mathsf{B}$ such that $A\xrightarrow{\alpha,r}A^{\prime}$, then we have
\[
I(A^{\prime})\xrightarrow{\alpha^{*},\langle \alpha,\rho\rangle-r}I(A),
\]
where $\alpha^{\ast}=-w_{0}(\alpha)$. In particular, $A\rightarrow A^{\prime}$
if and only if $I(A^{\prime})\rightarrow I(A)$.
\end{lemma}

\begin{proof}
We have
\[
I(A^{\prime})=t_{\rho}\circ w_{0}\circ s_{\alpha,r}(A)=t_{\rho}s_{w_{0}%
(\alpha),r}w_{0}(A)=s_{w_{0}(\alpha),r+\langle\rho,w_{0}(\alpha)\rangle
}t_{\rho}w_{0}(A).
\]
Since $w_{0}(\alpha)=-\alpha^{\ast}$ and $s_{\alpha,r}=s_{-\alpha,-r}$ for any
root $\alpha$, we get
\begin{equation}
I(A)=s_{\alpha^{\ast},\langle\alpha,\rho\rangle-r}I(A^{\prime}).
\label{relationReflection}%
\end{equation}
Moreover, we can write
\[
\langle I(v),\alpha^{\ast}\rangle=\langle t_{\rho}\circ w_{0}(v),\alpha^{\ast
}\rangle=\langle\rho,\alpha^{\ast}\rangle+\langle v,w_{0}(\alpha^{\ast
})\rangle=\langle\rho,\alpha^{\ast}\rangle-\langle v,\alpha\rangle
\]
where $v$ is the center of $A$. Since $\langle\rho,\alpha\rangle=\langle
\rho,\alpha^{\ast}\rangle$, this yields
\[
\langle I(v),\alpha^{\ast}\rangle=\langle\rho,\alpha\rangle-\langle
v,\alpha\rangle.
\]
By hypothesis, $r-1<\langle v,\alpha\rangle<r$, thus
\[
\langle\rho,\alpha\rangle-r<\langle I(v),\alpha^{\ast}\rangle<\langle
\rho,\alpha\rangle-r+1.
\]
The last inequalities together with \eqref{relationReflection} implies that
\[
I(A^{\prime})\xrightarrow{\alpha^{*},\langle \alpha,\rho\rangle-r}I(A).
\]

\end{proof}

\begin{lemma}
\label{stabilize_with_coeff} Suppose that $A,A^{\prime}$ are two elements of
$\mathsf{B}$ such that $A\xrightarrow{\alpha,r}t_{\Lambda_{i}}A^{\prime}$.
Then, we have $\alpha=\alpha_{i},r=1$, and
\[
I(A^{\prime})\xrightarrow{\alpha_{i^{*}},1}t_{\Lambda_{i^{\ast}}}I(A).
\]
In particular, $A\rightarrow t_{\Lambda_{i}}A^{\prime}$ if and only if
$I(A^{\prime})\rightarrow t_{\Lambda_{i^{\ast}}}I(A)$.
\end{lemma}

\begin{proof}
Since $A\xrightarrow{\alpha,r}t_{\Lambda_{i}}A^{\prime},$ the alcove
$s_{\alpha,r}(A)$ does not belong to $\mathsf{B}$, but belongs to the dominant
Weyl chamber. Also $\mathsf{B}$ is delimited by the affine hyperplanes
$H_{\alpha_{i},0}$ and $H_{\alpha_{i},1}$ for $1\leq i\leq k$, this implies
that $r=1$ and $\alpha$ is a simple root. Since $t_{-\Lambda_{i}}s_{\alpha
,1}(A)$ is contained in the dominant Weyl chamber, this yields that
$\alpha=\alpha_{i}$. Let $v$ and $v^{\prime}$ be the centers of $A$ and
$A^{\prime}$, respectively.

Then, $v=s_{\alpha,r}\circ t_{\Lambda_{i}}v^{\prime}$ for $s_{\alpha
,r}(v)=t_{\Lambda_{i}}v^{\prime}$. Using \eqref{commutRelation} we so derive%
\begin{multline*}
I(v)=t_{\rho}\circ w_{0}\circ s_{\alpha,r}\circ t_{\Lambda_{i}}(v^{\prime
})=t_{\rho}\circ w_{0}\circ t_{\Lambda_{i}}\circ s_{\alpha,r-\langle
\alpha,\Lambda_{i}\rangle}(v^{\prime})\\
=t_{\rho}\circ t_{w_{0}(\Lambda_{i})}\circ s_{w_{0}(\alpha),r-\langle
\alpha,\Lambda_{i}\rangle}\circ w_{0}(v^{\prime})=t_{w_{0}(\Lambda_{i})}\circ
s_{w_{0}(\alpha),r-\langle\alpha,\Lambda_{i}\rangle+\langle\rho,w_{0}%
(\alpha)\rangle}\circ t_{\rho}\circ w_{0}(v^{\prime})\\
=t_{w_{0}(\Lambda_{i})}\circ s_{w_{0}(\alpha),r-\langle\alpha,\Lambda_{i}%
+\rho\rangle}I(v^{\prime})
\end{multline*}
where $\langle\rho,w_{0}(\alpha)\rangle=\langle w_{0}(\rho),\alpha
\rangle=-\langle\rho,\alpha\rangle$ for the last equality. From the equality
$w_{0}(\Lambda_{i})=-\Lambda_{i^{\ast}}$, we get
\[
t_{\Lambda_{i^{\ast}}}I(v)=s_{\alpha^{\ast},\langle\alpha,\Lambda_{i}%
+\rho\rangle-r}I(v^{\prime}).
\]
Finally, we have $\langle I(v^{\prime}),\alpha^{\ast}\rangle=\langle
\rho,\alpha^{\ast}\rangle-\langle v^{\prime},\alpha\rangle=\langle\rho
,\alpha\rangle-\langle v^{\prime},\alpha\rangle$, so that the hypothesis
$r<\langle t_{\Lambda_{i}}v^{\prime},\alpha\rangle<r+1$ gives
\begin{align*}
r-\langle\Lambda_{i},\alpha\rangle &  <\langle v^{\prime},\alpha
\rangle<r+1-\langle\Lambda_{i},\alpha\rangle\text{ and}\\
\langle\rho+\Lambda_{i},\alpha\rangle-r-1  &  <\langle I(v^{\prime}%
),\alpha^{\ast}\rangle<\langle\rho+\Lambda_{i},\alpha\rangle-r.
\end{align*}
So
\[
I(w^{\prime}%
)\xrightarrow{\alpha^{*},\langle \alpha,\Lambda_{i}+\rho\rangle-r}t_{\Lambda
_{i^{\ast}}}I(w).
\]
Since $\alpha=\alpha_{i}$ and $r=1$, $\alpha^{\ast}=\alpha_{i}^{\ast}$ and
\[
\langle\alpha,\Lambda_{i}+\rho\rangle-r=\langle\alpha_{i},\Lambda_{i}%
+\rho\rangle-1=1.
\]

\end{proof}

\bigskip

Recall from Section \ref{subset_Lattices} that $\mathcal{B}_{k}$ is the Hasse
diagram for the weak Bruhat order on affine Grassmannian elements. We also
have a bijection which associates to $\lambda\in\mathcal{B}_{k}$ its
corresponding affine Grassmannian element $w_{\lambda}$. Let $1\leq a\leq k$.
Since the multiplication of $s_{\lambda}^{(k)}$ by $s_{R_{a}}$ is simply
$s_{\lambda\cup R_{a}}^{(k)}$, there exists a map $T_{a}$ on the set of
alcoves in the Weyl chamber such that $T_{a}(A_{w_{\lambda}})=A_{w_{\lambda
\cup R_{a}}}$. By \cite{BBTZ}, interpreting $k$-Schur functions as elements of
the affine nilCoxeter algebra yields that $T_{a}$ coincides with the
translation $t_{\Lambda_{a}}$ on the alcoves of the dominant Weyl chamber. In
particular, the partition $\lambda$ is irreducible if and only $A_{w_{\lambda
}}$ belongs to $\mathsf{B}$. By the definition of the matrix $\Phi
_{(r_{1},\ldots,r_{k})}$ in Section \ref{subsec_MatrixFI} we have
$\Phi_{(r_{1},\ldots,r_{k})}(\lambda,\mu)=1$ if $\lambda$ and $\mu$ are two
irreducible partitions such that $\lambda\rightarrow\mu$ on $\mathcal{B}_{k}$,
and $\Phi_{(r_{1},\ldots,r_{k})}(\lambda,\mu)=r_{a_{1}}+\cdots+r_{a_{s}}$ if
and only if $\lambda\rightarrow(\mu\cup R_{a_{1}}),\ldots,\lambda
\rightarrow(\mu\cup R_{a_{s}})$ on $\mathcal{B}_{k}$.

\begin{proposition}
\label{PhiSymmetryGrass} There exists an involutive permutation matrix
$\mathsf{I}$ such that
\[
\mathsf{I}\Phi_{(r_{1},\ldots,r_{k})}\mathsf{I}=\Phi_{(r_{k},\ldots,r_{1}%
)}^{t}%
\]
for any $(r_{1},\dots,r_{k})\in\mathbb{R}_{\geq0}^{k}$.
\end{proposition}

\begin{proof}
Let us write $\boldsymbol{\Phi=\Phi}_{(R_{1},\ldots,R_{k})}$ and
$\boldsymbol{\Phi}_{(R_{k},\ldots,R_{1})}$ for the matrix $\boldsymbol{\Phi}$
in which each $R_{a}$ is flipped in $R_{k-a+1}$. We get that
\begin{equation}
\left\{
\begin{matrix}
\boldsymbol{\Phi}_{(R_{1},\ldots,R_{k})}(\lambda,\mu)=1\Longleftrightarrow
w_{\lambda}\rightarrow w_{\mu}\\
\boldsymbol{\Phi}_{(R_{1},\ldots,R_{k})}(\lambda,\mu)=s_{R_{a_{1}}}%
+\cdots+s_{R_{a_{s}}}\Longleftrightarrow w_{\lambda}\rightarrow t_{\Lambda
_{a_{1}}}w_{\mu},\ldots,w_{\lambda}\rightarrow t_{\Lambda_{a_{s}}}w_{\mu}.
\end{matrix}
\right.  \label{PhiGrassmannian}%
\end{equation}
By Lemma \ref{Lem_IonB}, it makes sense to consider the matrix $\mathsf{I}$ of
the restriction of the involution $I$ on the basis $\mathcal{I}$. By Lemma
\ref{stabilize_without_coeff} and \eqref{PhiGrassmannian}, $\boldsymbol{\Phi
}_{(R_{1},\ldots,R_{k})}(\lambda,\mu)=1$ if and only if $\boldsymbol{\Phi
}_{(R_{k},\ldots,R_{1})}^{t}(I(\mu),I(\lambda))=1$, and $\boldsymbol{\Phi
}_{(R_{1},\ldots,R_{k})}=s_{R_{a_{1}}}+\cdots+s_{R_{a_{s}}}$ if and only if
$\boldsymbol{\Phi}_{(R_{1},\ldots,R_{k})}^{t}(I(\mu),I(\lambda
))=s_{R_{k+1-a_{1}}}+\cdots+s_{R_{k+1-a_{s}}}$. Since nonzero entries of
$\boldsymbol{\Phi}_{(R_{1},\ldots,R_{k})}$ are either $1$ or a sum $R_{i_{1}%
}+\cdots+R_{i_{s}}$ with $1\leq i_{j}\leq k$, this shows that $\mathsf{I}%
\boldsymbol{\Phi}_{(R_{1},\ldots,R_{k})}\mathsf{I}=\boldsymbol{\Phi}%
_{(R_{k},\ldots,R_{1})}^{t}$ and thus by specializing $\mathsf{I}\Phi
_{(r_{1},\ldots,r_{k})}\mathsf{I}=\Phi_{(r_{k},\ldots,r_{1})}^{t}.$
\end{proof}

The matrix $\Phi_{(r_{1},\ldots,r_{k})}$ exhibits thus two symmetries relating
to the $k$-conjugation and the involution $I$.

\begin{proposition}
\label{good_Symmetry} The matrices $\mathsf{I}$ and $\Omega$ commute, and
\[
(\mathsf{I}\Omega)\Phi_{(r_{1},\ldots,r_{k})}(\mathsf{I}\Omega)^{-1}%
=(\mathsf{I}\Omega)\Phi_{(r_{1},\ldots,r_{k})}(\mathsf{I}\Omega)=\Phi
_{(r_{1},\ldots,r_{k})}^{t}.
\]

\end{proposition}

\begin{proof}
In order to show that $\Omega$ and $\mathsf{I}$ commute, it suffices to show
that the involutions $I$ and $\omega_{k}$ commute at the level of their action
on alcoves in the dominant Weyl chamber. On the one hand, $I$ is the operator
$t_{\rho}w_{0}$. On the other hand, $\omega_{k}$ sends the $(k+1)$-core
associated to $\lambda$ to its usual conjugate. Hence, if we have the reduced
decomposition $w_{\lambda}=s_{i_{1}}s_{i_{2}}\cdots s_{i_{k}}$, we get the
reduced decomposition $w_{\omega(\lambda)}=s_{i_{1}^{\ast}}\cdots
s_{i_{k}^{\ast}}.$ This implies that the action of $\omega_{k}$ on the alcoves
coincide with that of $-w_{0}$ which commutes with $t_{\rho}w_{0}$ because
\[
t_{\rho}\circ-id=t_{\rho}w_{0}(-w_{0})=-w_{0}t_{\rho}w_{0}=-id\circ t_{-\rho
}.
\]
So we have $\mathsf{I}\Omega=\Omega\mathsf{I}$. The second part of the
proposition is a direct consequence of Proposition \ref{PhiSymmetryGrass} and
(\ref{PhiConjugation}).
\end{proof}

\subsection*{Drift under harmonic measures}

Let $\mathcal{A}_{k}$ be the set of alcoves in the dominant Weyl chamber. We
denote by $\Gamma_{f}(\mathcal{A}_{k})$ the set of reduced finite alcove paths
which start at $A^{(0)}$ and remain in the dominant Weyl chamber. For any $A$
in $\mathcal{A}_{k}$, write $\lambda_{A}\in\mathcal{B}_{k}$ its corresponding
$k$-bounded partition. Conversely recall that for any $\lambda\in
\mathcal{B}_{k}$, $A_{w_{\lambda}}\in\mathcal{A}_{k}$ is the alcove associated
to $\lambda$. Let $\varphi$ be an extremal harmonic measure on $\mathcal{B}%
_{k}$ associated to $\vec{r}=(r_{1},\ldots,r_{k})\in\mathbb{R}^{k}$, and let
$(A_{n})_{n\geq1}$ be the central Markov chain on $\mathcal{A}_{k}$ defined in
\S \ \ref{definitionCentral_Markov}. By considering for each $n\geq1$ the
center $v_{n}$ of the alcove $A_{n}$, we get a genuine random walk
$(v_{n})_{n\geq1}$ on $\mathbb{R}^{k}$. Our goal is now to prove the law of
large numbers for this random walk. This will be obtained by using the matrix
$\Phi=\Phi_{r_{1},\ldots,r_{k}}$ and a reduced version of the walk
$(v_{n})_{n\geq0}$. For simplicity, we will assume that $\Phi$ is irreducible.
Nevertheless, by continuity arguments, Theorem \ref{result_drift} below also
holds in full generality.

Observe that $1$ is the maximal eigenvalue of $\Phi$ for $\varphi$ is assumed
extremal harmonic. We denote by $X$ the corresponding left eigenvector of
$\Phi$ normalized so that $X(\emptyset)=1$. Let $\widehat{X}$ be the right
eigenvector also for the eigenvalue $1$ normalized so that $(X,\widehat{X})=1$
(here $(\cdot,\cdot)$ is the usual scalar product on vectors).

Let $\mathcal{M}_{k}$ be the multigraph with set of vertices $\mathsf{B}$ such
that for each affine reflection $s_{\alpha,r}$, we have an edge between $A$
and $A^{\prime}$ when $A^{\prime}=s_{\alpha,r}A$ or $t_{\Lambda_{i}}A^{\prime
}=s_{\alpha,r}A$. We color each edge $e$ by colors in $\{0,1,\ldots,k\}$ so
that $c(e)=i$ if $\alpha=\alpha_{i}$ is simple and with $c(e)=0$ otherwise
\footnote{Observe that $\mathcal{M}_{k}$ is the graph with adjacency matrix
$\Phi$ except that each arrow with weight $r_{i_{1}}+\cdots+r_{i_{m}}$ is
split in $m$ arrows with weights $r_{i_{1}},\ldots,r_{i_{m}}$.}. Let
$(\widetilde{A}_{n})_{n\geq1}$ be the Markov chain on the graph $\mathcal{M}%
_{k}$ starting on $A^{(0)}$ with transition probabilities $\widetilde
{\mathbb{P}}(A\xrightarrow{e}A^{\prime})=r_{c(e)}\frac{X(\lambda_{A^{\prime}%
})}{X(\lambda_{A})}$, with the convention $r_{0}=1$. Note that $(\widetilde
{A}_{n})_{n\geq1}$ is indeed a random walk, since
\[
\sum_{\substack{e,A^{\prime}\\A\text{ gives }A^{\prime}\text{ through }%
e}}r_{c(e)}\frac{X(\lambda_{A^{\prime}})}{X(\lambda_{A})}=\sum_{A,A^{\prime}%
}\Phi_{\lambda_{A},\lambda_{A^{\prime}}}\frac{X(\lambda_{A^{\prime}}%
)}{X(\lambda_{A})}=\frac{X(\lambda_{A})}{X(\lambda_{A})}=1.
\]
The weight $\mathrm{wt}(\gamma)$ of a path $\gamma=A_{0}%
\xrightarrow{e_{1}}A_{1}\xrightarrow{e_{2}}\dots\xrightarrow{e_{n}}A_{n}$ is
defined by $\mathrm{wt}(\gamma)=\sum_{i=1}^{n}\Lambda_{c(e_{i})}$, with the
convention that $\Lambda_{0}=0$. We denote by $\ell$ the associated length
function. Let $\Gamma_{f}(\mathcal{M}_{k})$ and $\Gamma_{f}(\mathcal{A}_{k})$
be respectively the sets of finite paths on $\mathcal{M}_{k}$ and
$\mathcal{A}_{k}$ starting at $A_{0}$.

We define $p:\Gamma_{f}(\mathcal{M}_{k})\longrightarrow\mathcal{A}_{k}$ by
$p(\gamma)=\gamma(n)+\mathrm{wt}(\gamma)$, where $n$ is the length of $\gamma
$, and we extend the map $p$ to a map $L:\Gamma_{f}(\mathcal{M}_{k}%
)\rightarrow\Gamma_{f}(\mathcal{A}_{k})$ where $L(\gamma)=(p(A_{0}%
),p(A_{0},A_{1}),\ldots,p(A_{0},\ldots,A_{n}))$. Let $M:\Gamma_{f}%
(\mathcal{A}_{k})\rightarrow\Gamma_{f}(\mathcal{M}_{k})$ be the map which
sends a path $(A_{0}\rightarrow A_{1}\cdots\rightarrow A_{n})$ to the path
$(\tilde{A}_{0}\xrightarrow{e_{1}}\tilde{A}_{1}\xrightarrow{e_{2}}\dots
\xrightarrow{e_{n}}\tilde{A}_{n})$, where $e_{i}$ is the unique edge from
$\tilde{A}_{i-1}$ to $\tilde{A}_{i}$ such that $c(e_{i})=j$ if $A_{i}%
=s_{\alpha_{j},k}A_{i-1}$ for some $k\in\mathbb{Z}_{>0}$ and $1\leq j\leq k$,
and $e_{i}$ is the unique edge from $\tilde{A}_{i-1}$ to $\tilde{A}_{i}$ with
color $0$ if $A_{i}=s_{\alpha,k}A_{i-1}$ with $\alpha$ non-simple and
$k\in\mathbb{Z}_{>0}$. It is easy to see that $LM=id_{\Gamma_{f}%
(\mathcal{A}_{k})}$ and $ML=id_{\Gamma_{f}(\mathcal{M}_{k})}$.

\begin{lemma}
The image of the Markov chain $(\tilde{A}_{n})_{n\geq0}$ through the map $L$
is exactly the Markov chain $(A_{n})_{n\geq0}$.
\end{lemma}

\begin{proof}
Let $\gamma$ be a finite path on $\mathcal{M}_{k}$ of weight $\mathrm{wt}%
(\gamma)$ and ending at $\tilde{A}$. By the Markov kernel defined above,
\[
\widetilde{\mathbb{P}}(\gamma)=r^{_{\mathrm{wt}(\gamma)}}X(\lambda_{\tilde{A}%
}),
\]
where $r^{\mathrm{wt}(\gamma)}=r_{1}^{\beta_{1}}\dots r_{k}^{\beta_{k}}$ when
$\mathrm{wt}(\gamma)=\beta_{1}\Lambda_{1}+\dots+\beta_{k}\Lambda_{k}$, with
$\beta_{i}\in\mathbb{Z}_{\geq0}$. Since $L(\gamma)$ ends at $p(\gamma
)=A+\mathrm{wt}(\gamma)$ and $X(\lambda)=\varphi(s_{\lambda}^{(k)})$ for any
$\lambda\in\mathcal{P}_{\mathrm{irr}}$, we have
\[
\widetilde{\mathbb{P}}(\gamma)=\mathbb{P}(L(\gamma)).
\]

\end{proof}

\bigskip

Recall that for any $n\geq0$, $v_{n}$ is the center of the alcove $A_{n}$.
Denote by $x_{i}(n)=\langle v_{n},\alpha_{i}\rangle$ the position of $v_{n}$
along the direction $\Lambda_{i}$.

\begin{lemma}
As $n$ goes to infinity,
\[
\frac{1}{n}x_{i}(n)\longrightarrow r_{i}\sum_{\substack{e:A\rightarrow
A^{\prime}\in\mathcal{M}_{k}\\c(e)=i}}\widehat{X}(\lambda_{A})X(\lambda
_{A^{\prime}}).
\]

\end{lemma}

\begin{proof}
Set $y_{i}(n)=\lfloor x_{i}(n)\rfloor$. Then,
\[
\lim_{n\rightarrow\infty}\frac{1}{n}x_{i}(n)=\lim_{n\rightarrow\infty}\frac
{1}{n}y_{i}(n).
\]
Let $N\geq1$ and $0\leq n\leq N$. Suppose that $y_{i}(n+1)-y_{i}(n)=1$. Since
$x_{i}(n)-y_{i}(n)>0$, we have $\langle v_{n},\Lambda_{i}\rangle
<y_{i}(n+1)<\langle v_{n+1},\Lambda_{i}\rangle$. Hence, the affine hyperplane
$H_{\alpha_{i},y_{i}(n+1)}$ separates the alcoves $A_{n}$ and $A_{n+1}$, and
thus $A_{n+1}=s_{\alpha_{i},y_{i}(n+1)}A_{n}$. Hence, $y_{i}(n+1)-y_{i}(n)=1$
if and only if the $n$-th edge of the path $M(A_{1},\ldots,A_{N})$ is colored
by $i$. Hence $y_{i}(N)$ is the number of arrows colored by $i$ in the
trajectory $M(A_{1},\ldots,A_{N})$. Since $M(A_{1},\ldots,A_{N})$ is an
irreducible random walk on $\mathcal{M}_{k}$, the ergodic theorem for random
walks on finite spaces yields that for each edge $e_{0}\in\mathcal{M}_{k}$
from $A$ to $A^{\prime}$,
\[
\frac{1}{N}\mathrm{card}(\{e\in M(A=A_{1},\ldots,A_{N}=A^{\prime}%
),e=e_{0}\}\xrightarrow[n\rightarrow \infty]{a.s}m(A)\widetilde{\mathbb{P}%
}(A\xrightarrow{e}A^{\prime}),
\]
where $m$ is the invariant measure on $\mathcal{M}_{k}$ with respect to
$\widetilde{\mathbb{P}}$. We have%
\[
\widetilde{\mathbb{P}}(\widetilde{A}_{n}=A^{\prime}\mid\widetilde{A}%
_{n-1}=A^{\prime})=\sum_{e\text{ from }A\text{ to }A^{\prime}}r_{c(e)}%
\frac{X(\lambda_{A^{\prime}})}{X(\lambda_{A})}=\Phi_{\lambda_{A}%
,\lambda_{A^{\prime}}}\frac{X(\lambda_{A^{\prime}})}{X(\lambda_{A})},
\]
thus the corresponding invariant measure is the unique vector $m$ such that
$\sum_{A\in\mathcal{M}_{k}}m(A)=1$ and
\[
\sum_{A}\Phi_{\lambda_{A},\lambda_{A^{\prime}}}\frac{X(\lambda_{A^{\prime}}%
)}{X(\lambda_{A})}m(A)=m(A^{\prime}).
\]
We get that $\left(  \frac{m(A)}{X(\lambda_{A})},A\in\mathcal{M}_{k}\right)  $
is a left eigenvector of $\Phi$ with eigenvalue $1$, thus is proportional to
$(\widehat{X}(\lambda_{A}),A\in\mathcal{M}_{k})$. In fact it is equal to
$\widehat{X}(\lambda_{A})$ since $m$ is a measure and $(X,\widehat{X})=1$ so
$m(A)=X(\lambda_{A})\widehat{X}(\lambda_{A}).$ This gives
\[
\frac{1}{N}\mathrm{card}(\{e\in M(v_{1},\ldots,v_{N}),e=e_{0}%
\}\xrightarrow[n\rightarrow \infty]{a.s}X(\lambda_{A})\widehat{X}(\lambda
_{A})r_{c(e_{0})}\frac{X(\lambda_{A^{\prime}})}{X(\lambda_{A})}=\widehat
{X}(\lambda_{A})X(\lambda_{A^{\prime}})r_{c(e_{0})}.
\]
Since $y_{i}(N)$ is the number of arrows colored by $i$ in the trajectory
$M(A_{1},\ldots,A_{N})$ we obtain%
\[
\frac{1}{N}y_{i}(N)\xrightarrow[n\rightarrow \infty]{a.s}r_{i}\sum
_{\substack{e:A\rightarrow A^{\prime}\\c(e)=i}}\widehat{X}(\lambda
_{A})X(\lambda_{A^{\prime}}).
\]

\end{proof}

\bigskip

For any alcove $A$, set $\overline{A}=\mathsf{I}\Omega(A)$.

\begin{theorem}
\label{result_drift} \ 

\begin{enumerate}
\item As $n$ goes to infinity, the normalized random walk $\left(  \frac{1}%
{n}v_{n}\right)  _{n\geq1}$ converges almost surely to a vector $v_{\varphi
}\in\mathbb{R}^{k}$.

\item Moreover for any $i=1,\ldots,k$ the coordinate of $v_{\varphi}$ on
$\Lambda_{i}$ satisfies
\[
v_{\varphi}(i)=\varphi\left(  \frac{s_{R_{i}}}{\sum_{A\in\mathsf{B}}%
s_{\lambda_{A}}^{(k)}s_{\lambda_{\overline{A}}}^{(k)}}\sum
_{\substack{e:A\rightarrow A^{\prime}\\c(e)=i}}s_{\lambda_{\overline{A}}%
}^{(k)}s_{\lambda_{A^{\prime}}}^{(k)}\right)
\]
which is a rational function on $\mathbb{R}^{k}$.
\end{enumerate}
\end{theorem}

\begin{proof}
The previous lemma proves the first part of the theorem. It also shows that
\[
v_{\varphi}(i)=r_{i}\sum_{\substack{e:A\rightarrow A^{\prime}\\c(e)=i}%
}\widehat{X}(\lambda_{A})X(\lambda_{A^{\prime}}).
\]
By Proposition \ref{good_Symmetry}, the coordinates of the vector $\widehat
{X}$ are such that $\widehat{X}(\lambda_{A})=\frac{1}{\nabla}X(\lambda
_{\mathsf{I}\Omega(A)})$ for any $A\in\mathsf{B}$ where $\nabla=\sum
_{A\in\mathsf{B}}X(\lambda_{A})X(\lambda_{\overline{A}})$. Since the
coordinates of $X$ are the $\varphi(s_{\lambda}^{(k)})$ with $\lambda$
irreducible, we can write%
\[
v_{\varphi}(i)=\frac{1}{\nabla}r_{i}\sum_{\substack{e:A\rightarrow A^{\prime
}\\c(e)=i}}\varphi(s_{\lambda_{\overline{A}}}^{(k)}s_{\lambda_{A^{\prime}}%
}^{(k)})=\varphi\left(  \frac{s_{R_{i}}}{\sum_{A\in\mathsf{B}}s_{\lambda_{A}%
}^{(k)}s_{\lambda_{\overline{A}}}^{(k)}}\sum_{\substack{e:A\rightarrow
A^{\prime}\\c(e)=i}}s_{\lambda_{\overline{A}}}^{(k)}s_{\lambda_{A^{\prime}}%
}^{(k)}\right)  .
\]
Proposition \ref{Prop_rational} below applied with $t=1$ will imply that
$v_{\varphi}(i)$ is indeed a rational function in $(r_{1},\ldots,r_{k})$ and
so $v_{\varphi}$ is rational in $(r_{1},\ldots,r_{k})$.
\end{proof}

\section{Some consequences}

\subsection{Limit formulas in the case $\varphi(\Delta)=0$}

\label{SubsecLim}For any $k$-irreducible partition $\kappa$, we know by
\S \ \ref{subsec_primitiveel} that there exists a polynomial $P_{\kappa}%
\in\mathbb{A}[T]$ such that
\begin{equation}
s_{\kappa}^{(k)}=\frac{P_{\kappa}(s_{1})}{\Delta} \label{FormPrim}%
\end{equation}
here $\Delta\in\mathbb{A}$ is the determinant of the transition matrix between
the bases $\{s_{(1)}^{a}\mid0\leq a\leq k!-1\}$ and $\mathcal{I}=\{s_{\kappa
}^{(k)}\mid\kappa\in\mathcal{P}_{\mathrm{irr}}\}$. For any morphism
$\varphi:\Lambda_{(k)}\rightarrow\mathbb{R}$ nonnegative on the $k$-Schur
functions and such that $\varphi(\Delta)\neq0$ we thus get%
\begin{equation}
\varphi(s_{\kappa}^{(k)})=\frac{\varphi(P_{\kappa})(\varphi(s_{1}))}%
{\varphi(\Delta)}. \label{FormSpe}%
\end{equation}
Moreover, $\varphi(P_{\kappa})$ and $\varphi(\Delta)$ are directly determined
by the values $r_{a}=\varphi(s_{R_{a}}),a=1,\ldots,k$ since $P_{\kappa}$ and
$\Delta$ belong to the subalgebra $\mathbb{A}$. Also $(\varphi(s_{1})$ is the
spectral radius of the matrix $\Phi=\varphi(\boldsymbol{\Phi})$.

Now assume that the morphism $\varphi$ associated to $\vec{r}$ is such that
$\varphi(\Delta)=0$. Then, we can consider a sequence $(\vec{r}_{n})_{n\geq0}$
in $U=\mathbb{R}_{\geq0}^{k}$ such that each morphism $\varphi_{n}:=f(\vec
{r}_{n})$ satisfies $\varphi_{n}(\Delta)\neq0$.\ By continuity of the map $f$
we then get for any $k$-irreducible partition $\kappa$%
\[
\varphi(s_{\kappa}^{(k)})=\lim_{n\rightarrow\infty}\frac{\varphi_{n}%
(P_{\kappa})(\varphi_{n}(s_{1}))}{\varphi_{n}(\Delta)}%
\]
so that the formulas (\ref{FormSpe}) extends by continuity. In particular we
then have $\varphi(P_{\kappa})(s_{1})=0$. Alternatively, one can consider for
any nonnegative real $s$, the sets $\overline{V}_{s}=\{\vec{h}\in\overline
{V}\mid h_{1}=s\}$ and $\overline{U}_{s}=g(\overline{V}_{s})$. For any
$\vec{r}\in\overline{U}_{s}$ such that $\Delta(\vec{r})=\varphi(\Delta)\neq0$
write $\widetilde{P}_{\kappa}^{s}(\vec{r})=\varphi(P_{\kappa})(s)$. We also
set $\varphi(s_{\kappa}^{(k)})=s_{\kappa}^{(k)}(\vec{r})$.

\begin{proposition}
\label{Prop_rational}For each irreducible $k$-bounded partition $\kappa$, the
function $\vec{r}\longmapsto s_{\kappa}^{(k)}(\vec{r})$ is continuous on
$\overline{U}_{s}$ and rational. We have%
\begin{equation}
s_{\kappa}^{(k)}(\vec{r})=\frac{\widetilde{P}_{\kappa}^{s}(\vec{r})}%
{\Delta(\vec{r})}. \label{S_kappa_sfixed}%
\end{equation}
In particular, the coordinates of $f$ are continuous rational functions on
each $\overline{U}_{s}$.
\end{proposition}

\subsection{The minimal boundary of $\mathcal{B}^{(3)}$}

For $k=3$, one can easily picture the domains $\overline{V}_{1}$. The
condition to get a positive morphism $\varphi$ indeed reduces to
$\varphi(s_{1})\geq0,\varphi(s_{2})=\varphi(h_{2})\geq0,\varphi(s_{3}%
)=\varphi(h_{3})\geq0,\varphi(s_{(1,1)})=\varphi(e_{2})\geq0,\varphi
(s_{(1,1,1)})=\varphi(e_{3})\geq0,\varphi(s_{(2,1)})\geq0,\varphi
(s_{(2,1,1)}^{(3)})\geq0$ and $\varphi(s_{(2,2)})\geq0$. We get moreover by a
simple computation
\[
s_{(2,1,1)}^{(3)}=s_{(2)}s_{(1,1)}%
\]
thus $\varphi(s_{(2,1,1)}^{(3)})\geq0$ does not add any new constraint. We
also have $\varphi(h_{1})=\varphi(s_{1})=1$ and the Jacobi-Trudi relations
$e_{2}=h_{1}^{2}-h_{2},e_{3}=h_{1}^{3}+h_{3}-2h_{2}h_{1},s_{(2,1)}=h_{2}%
h_{1}-h_{3}$ and $s_{(2,2)}=h_{2}^{2}-h_{3}h_{1}$. Now by using that
$\varphi(s_{1})=1$ one can see that the previous inequalities are equivalent
to%
\begin{equation}
h_{1}=1,\ 0\leq h_{2}\leq1,\ 0\leq h_{3}\leq h_{2}^{2}\text{ and }2h_{2}%
-h_{3}\leq1. \label{condi3}%
\end{equation}
By setting $x=\varphi(h_{2})$ and $y=\varphi(h_{3})$. This gives the domain
$\overline{V}_{1}=\partial\mathcal{H}^{+}(\mathcal{B}_{3})$ delimited in the
picture below by the $x$ abscissa, the blue line and the red parabola.

\begin{figure}
\begin{center}
\includegraphics[
width=7.479500cm,
height=5.4191cm,
width=7.4795cm]{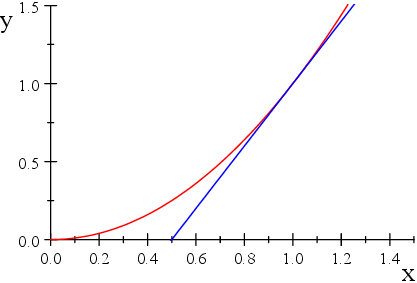}%

\end{center}
\caption{Region $\overline{V}_{1}$ in $x=h_{2}$ and $y=h_{3}$ coordinates delimited by the three curves $h_{3}=0$, $s_{(2,2)}=0$ and $e_{3}=0$.}
\end{figure}

\begin{remark}
\label{rem_badRest}If we consider the points of $\overline{V}_{1}%
=\partial\mathcal{H}^{+}(\mathcal{B}_{3})$ such that $h_{3}=0$, we get the
domain $\{(1,h_{2},0)\mid h_{2}\in\lbrack0,\frac{1}{2}])$.\ From example
\ref{Ex_k=2}, one sees that its projection in $\mathbb{R}^{2}$ is only
strictly contained in $\partial\mathcal{H}^{+}(\mathcal{B}_{2})=\{(1,h_{2}%
)\mid h_{2}\in\lbrack0,1]\}$ (see \S \ \ref{Subsec_PojLim}).
\end{remark}

\subsection{Minimal boundary of $\mathcal{B}_{k}$}

By homogeneity of the Schur functions, one gets that for any $\vec{r}%
=(r_{1},r_{2},\ldots,r_{k})\in\mathbb{R}_{\geq0}^{k}$ and any positive real
$t$%
\begin{equation}
f(t^{k}r_{1},t^{2(k-1)}r_{2},\ldots,t^{k}r_{k})=(th_{1},\ldots,t^{k}h_{k}).
\label{Rela}%
\end{equation}
Also with the notation of \S \ \ref{SubsecLim}, we obtain that $\partial
\mathcal{H}^{+}(\mathcal{B}_{k})=\overline{V}_{1}$ is homeomorphic to
$\overline{U}_{1}$. It follows from (\ref{Rela}) that for each nonnegative
real $s$, the sets $\overline{U}_{s}$ and $\overline{V}_{s}$ are completely
determined by $\overline{U}_{1}$ and $\overline{V}_{1}$, respectively. Also,
we can associate to any element $\vec{r}\in\mathbb{R}_{\geq0}^{k}$ the element
in $\partial\mathcal{H}^{+}(\mathcal{B}_{k})$ obtained by computing $\vec
{h}=f(\vec{r})$ and next renormalizing it according to (\ref{Rela}) so that
its first coordinate becomes equal to $1$. We also have the following
description of the minimal boundary:

\begin{proposition}
$\partial\mathcal{H}^{+}(\mathcal{B}_{k})$ is homeomorphic to $\mathcal{S}%
_{k}=\{(r_{1},\ldots,r_{k})\in\mathbb{R}_{\geq0}^{k}\mid r_{1}+\cdots
+r_{k}=1\}$.
\end{proposition}

\begin{proof}
We already know that $\partial\mathcal{H}^{+}(\mathcal{B}_{k})=\overline
{V}_{1}$ is homeomorphic to $\overline{U}_{1}$. Also any $\vec{r}%
=(r_{1},\ldots,r_{k})$ in $\overline{U}_{1}$ is nonzero. There thus exists a
unique positive real $t(\vec{r})$ such that $t^{k}r_{1}+t^{2(k-1)}r_{2}%
+\cdots+t^{k}r_{k}=1$.\ This follows from the fact that the polynomial
function $p(t)=t^{k}r_{1}+t^{2(k-1)}r_{2}+\cdots+t^{k}r_{k}$ strictly
increases on $\mathbb{R}_{>0}$ with $p(0)=0$ and $\mathrm{lim}_{t\rightarrow
+\infty}=+\infty$. Then, $t(\vec{r})$ is the unique real root of the
polynomial $p(T)-1$. The function $t:\vec{r}\rightarrow t(\vec{r})$ is
continuous on $\overline{U}_{1}$, therefore the function $u:\overline{U}%
_{1}\rightarrow\mathcal{S}_{k}$ defined by%
\[
u(r_{1},\ldots,r_{k})=(t(\vec{r})^{k}r_{1},t(\vec{r})^{2(k-1)}r_{2}%
,\ldots,t(\vec{r})^{k}r_{k})
\]
is well-defined and continuous. If $u(\vec{r})=u(\vec{R})$ with $\vec{r}$ and
$\vec{R}$ in $\partial\mathcal{H}^{+}(\mathcal{B}_{k})$, we have by applying
$f:$%
\[
f(u(\vec{r}))=(t(\vec{r})1,t(\vec{r})^{2}h_{2},\ldots,t(\vec{r})^{k}%
h_{k})=(t(\vec{R})1,t(\vec{R})^{2}h_{2},\ldots,t(\vec{R})^{k}h_{k}%
)=f(u(\vec{R})).
\]

Thus $t(\vec{r})=t(\vec{R})$ and we get $\vec{r}=\vec{R}$ so that $u$ is
injective.\ Now given any $\vec{r}=(r_{1},\ldots,r_{k})\in\mathcal{S}_{k}$,
there exists a positive real $s$ such that $\vec{r}_{s}=(s^{k}r_{1}%
,s^{2(k-1)}r_{2},\ldots,s^{k}r_{k})$ belongs to $\partial\mathcal{H}%
^{+}(\mathcal{B}_{k})$. We then have $u(\vec{r}_{s})=\vec{r}$.
\end{proof}

By observing that $\Lambda_{(k)}=\mathbb{R}[h_{1},\ldots,h_{k}]=\mathbb{R}%
[e_{1},\ldots,e_{k}]$ is in fact isomorphic to the algebra $\Lambda\lbrack
X_{1},\ldots,X_{k}]$ of symmetric polynomials in $k$ variables $X_{1}%
,\ldots,X_{k}$ over $\mathbb{R}$, we can also get informations on the values
taken by these variables for each point of $\partial\mathcal{H}^{+}%
(\mathcal{B}_{k}).$ For any $r=1,\ldots,k$, write for short $E_{r}%
=\frac{\widetilde{P}_{(1^{k})}^{1}(\vec{r})}{\Delta(\vec{r})}$.\ Each $E_{r}$
is a rational function on $\overline{U}_{1}$ which associates to an element of
$\overline{U}_{1}$ the value of $\varphi(e_{r})$ for the associated morphism
$\varphi$.

\begin{proposition}
For each $\vec{h}\in\partial\mathcal{H}^{+}(\mathcal{B}_{k})$, there exists a
unique $\vec{x}=(x_{1},\ldots,x_{k})\in\mathbb{C}^{k}$ such that the
associated morphism $\varphi:\Lambda_{(k)}\rightarrow\mathbb{R}$, nonnegative
on the $k$-Schur functions, coincides with the specialization%
\[
\varphi(P(X_{1},\ldots,X_{r}))=P(x_{1},\ldots,x_{r}).
\]
Moreover $\vec{x}$ is determined by the roots of the polynomial%
\[
\zeta(T)=\prod_{r=1}^{k}(1+Tx_{i})=1+t+\sum_{r=2}^{k-1}E_{r}T^{r-1}+r_{k}T^{k}%
\]
where $E_{1},\ldots,E_{r}$ are rational continuous functions on $\overline
{U}_{1}$.
\end{proposition}

\begin{example}
\ 

\begin{enumerate}
\item For $k=2$, we have $E_{1}=1$ and $E_{2}=r_{2}$ so that
\[
\zeta(T)=1+t+t^{2}r_{2}.
\]

\item For $k=3$, we get by resuming Example \ref{Examk=3} and using the
equality $\Xi(1)=1-2\left(  r_{1}+r_{3}\right)  -4r_{2}+\left(  r_{1}%
-r_{3}\right)  ^{2}=0$.%
\[
E_{1}=1\text{ and }E_{2}=\frac{1}{2}(r_{3}-r_{1}+1).
\]
This gives%
\[
\zeta(T)=1+T+\frac{1}{2}(r_{3}-r_{1}+1)T^{2}+r_{3}T^{3}.
\]
In that simple case we get in fact polynomial functions independent of $r_{2}$.
\end{enumerate}
\end{example}

\begin{remark}
The previous proposition does not mean that $\partial\mathcal{H}%
^{+}(\mathcal{B}_{k})$ is parametrized by the roots of all the polynomials
$\zeta(T).\ $This is only true for the roots of the polynomials $\zeta(T)$
corresponding to a point in $\overline{U}_{1}$.
\end{remark}

\subsection{Embedding and projective limit of the minimal boundaries}

\label{Subsec_PojLim}By Proposition \ref{Prop_positiveExp}, each morphism
$\varphi:\Lambda_{(k+1)}\rightarrow\mathbb{R}$ nonnegative on the
$(k+1)$-Schur functions yields by restriction to $\Lambda_{(k)}\subset
\Lambda_{(k+1)}$ a morphism nonnegative on the $k$-Schur functions. Here we
use the natural embedding $\Lambda_{(k)}\subset\Lambda_{(k+1)}$ corresponding
to the specialization $h_{k+1}=0$. Unfortunately, this will not give us a
projection of $\partial\mathcal{H}^{+}(\mathcal{B}_{k+1})$ on $\partial
\mathcal{H}^{+}(\mathcal{B}_{k})$ (see Remark \ref{rem_badRest}).
Nevertheless, we can define such a projection $\pi_{k}:\partial\mathcal{H}%
^{+}(\mathcal{B}_{k+1})\rightarrow\partial\mathcal{H}^{+}(\mathcal{B}_{k})$ by
setting%
\[
\pi_{k}(h_{1},\ldots,h_{k},h_{k+1})=\pi_{k}\circ f(r_{1},\ldots,r_{k}%
,r_{k+1})=f(r_{1},\ldots,r_{k})
\]
where $f(r_{1},\ldots,r_{k},r_{k+1})=(h_{1},\ldots,h_{k},h_{k+1})$. This
indeed yields a surjective map since for any $(h_{1}^{\prime},\ldots
,h_{k}^{\prime})\in\partial\mathcal{H}^{+}(\mathcal{B}_{k})$, we can set
$(h_{1}^{\prime},\ldots,h_{k}^{\prime})=\pi_{k}\circ f(r_{1}^{\prime}%
,\ldots,r_{k}^{\prime},0)$ where $(r_{1}^{\prime},\ldots,r_{k}^{\prime
})=g(h_{1}^{\prime},\ldots,h_{k}^{\prime})$.

\begin{proposition}
\ 

\begin{enumerate}
\item The map $\pi_{k}$ is continuous and surjective from $\partial
\mathcal{H}^{+}(\mathcal{B}_{k+1})$ to $\partial\mathcal{H}^{+}(\mathcal{B}%
_{k})$.

\item The inverse limit $\underleftarrow{\mathrm{lim}}\mathcal{B}_{k}$ is
homeomorphic to the minimal boundary of the ordinary Young lattice, that is to
the Thoma simplex.
\end{enumerate}
\end{proposition}

\subsection{Rietsch parametrization of Toeplitz matrices}

Consider the variety $T_{\geq0}\subset\mathbb{R}_{>0}^{k}$ of totally
nonnegative unitriangular Toeplitz $(k+1)\times(k+1)$ matrices
\[
M=\left[
\begin{array}
[c]{cccccc}%
1 &  &  &  &  & \\
h_{1} & 1 &  &  &  & \\
\vdots & h_{1} & \ddots &  &  & \\
\vdots & \vdots & \ddots & \ddots &  & \\
h_{k-1} & \vdots & \vdots & \ddots & \ddots & \\
h_{k} & h_{k-1} & \cdots & \cdots & h_{1} & 1
\end{array}
\right]  .
\]
The set $T_{>0}$ of totally positive unitriangular Toeplitz $(k+1)\times(k+1)$
matrices is defined as the subset of $T_{\geq0}$ of matrices $M$ whose minors
with no row and no column in the upper part of $M$ are positive. By Theorem
3.2.1 in \cite{BFZ}, $M$ is totally positive if and only if for $a=1,\ldots
,k$, the $a\times a,$ initial minors obtained by selecting $a$ rows of $M$
arbitrary and then the first $a$ columns of $M$ are positive.

\begin{lemma}
\ \label{Lemma_closed}

\begin{enumerate}
\item The previous initial minors are equal to Schur functions $s_{\lambda}$,
where the maximal hook of the partition $\lambda$ has length less or equal to
$k$.

\item We have $\overline{T}_{>0}=T_{\geq0}$ that is, each totally nonnegative
unitriangular Toeplitz matrix is the limit of a sequence of totally positive
unitriangular Toeplitz matrices.
\end{enumerate}

\begin{proof}
Let $L=\{i_{1},\ldots,i_{a}\}$ be a subset of $\{1,\ldots,k\}$ such that
$i_{1}<\cdots<i_{k}$ and consider the minor $\Delta_{L}$ corresponding to the
determinant of the submatrix $M_{L\times\lbrack1,a]}$. The diagonal of
$M_{L\times\lbrack1,a]}$ is $(h_{i_{1}},h_{i_{2}-1},\cdots,h_{i_{k}-k+1})$
where $i_{k}-k+1\geq\cdots\geq i_{2}-1\geq i_{1}$. Thus, by using the
Jacobi-Trudi formula we have $\Delta_{L}=s_{(i_{k}-k+1,\ldots,i_{2}-1,i_{1})}%
$. The maximal hook length of the partition $\lambda=(i_{k}-k+1,\ldots
,i_{2}-1,i_{1})$ is equal to $(i_{k}-k+1)+(k-1)=i_{k}\leq k$ which proves
assertion 1$.$

To get Assertion 2, consider $M\in T_{\geq0}$ and $U\in T_{>0}$. For any real
$t>0$ let $U(t)$ be the matrix obtained by replacing each real $h_{a}$ by
$t^{a}h_{a}$ in $U$. Then $U(t)$ belongs to $T_{>0}.\ $Indeed, with the
previous notation, if the minor $\Delta_{L}$ associated to $U$ is equal to the
Schur function $s_{\lambda}$, then the corresponding minor in $U(t)$ is equal
to $t^{\left\vert \lambda\right\vert }s_{\lambda}$.\ The set $T_{\geq0}$ is
stable by matrix multiplication and we moreover get from Proposition 10 in
\cite{FZ} that the product matrix $U(t)M$ is totally positive. Since $U(t)$
tends to the identity matrix when $t$ tends to $0$, we obtain that $U(t)M$
tends to $M$ as desired.
\end{proof}
\end{lemma}

Observe in particular that for any $a=1,\ldots,k$, the initial minor
$\Delta_{\lbrack k-a+1,k]}$ gives the value $r_{a}$ of the rectangle Schur
function $s_{R_{a}}$ evaluated in $(h_{1},\ldots,h_{k})$. In \cite{Ri},
Rietsch obtained the following parametrization of $T_{\geq0}$ by using the
quantum cohomology of partial flag varieties.

\begin{theorem}
The map%
\[
\left\{
\begin{array}
[c]{c}%
T_{\geq0}\rightarrow\overline{U}\\
(h_{1},\ldots,h_{k})\longmapsto(r_{1},\ldots,r_{k})
\end{array}
\right.
\]
is a homeomorphism.
\end{theorem}

We now reprove this theorem from our preceding results.

\begin{theorem}
\label{Th_RiRef}We have $T_{>0}=V$ and $T_{\geq0}=\overline{V}$, in particular
the map $g:T_{\geq0}\rightarrow\overline{U}$ is a homeomorphism$.$
\end{theorem}

\begin{proof}
Observe first we have $V\subset T_{>0}$. Indeed we know that each $k$-Schur
function $s_{\lambda}^{(k)}$ evaluated in $\vec{h}=(h_{1},\ldots,h_{k})$ in
$V$ is positive. This is in particular true when $\lambda$ is a partition with
maximal hook length less or equal to $k$ but then, we get by Assertion 1 of
the previous Lemma that the associated Toeplitz matrix is totally positive
because such $k$-Schur functions coincide with ordinary Schur functions. Next
consider a sequence $\vec{h}_{n},n\geq0$ in $V$ which converges to a limit
$\vec{h}\in T_{>0}$. Since $\vec{h}\in T_{>0}$, each $r_{a}=\Delta_{\lbrack
k-a+1,k]}(\vec{h}),a=1,\ldots,k$ is positive. Thus $\vec{r}=(r_{1}%
,\ldots,r_{k})$ belongs to $U$. Now $\vec{h}$ belongs to $\overline{V}$ and we
have $g(\vec{h})=\vec{r}$ by definition of $g$. Theorem \ref{Th_UVhomeo} then
implies that $\vec{h}\in V$ so $V$ is closed in $T_{>0}$. Now $V$ is open in
$T_{>0}$ because each $\vec{h}\in V$ admits a neighborhood contained in
$V\subset T_{>0}$ ($V$ is an intersection of open subsets by definition). We
also have that $T_{>0}$ is connected (see for example the proof of Proposition
12.2 in \cite{Ri}). So $V$ is nonempty both open and closed in $T_{>0}$ and we
therefore have $T_{>0}=V$. The second assertion of Lemma \ref{Lemma_closed}
then gives $T_{\geq0}=\overline{T}_{>0}=\overline{V}$.
\end{proof}

\bigskip

\begin{remarks}
\ \label{Rema_redtest}

\begin{enumerate}
\item Since $T_{>0}=V,$ we get by using the initial minors of $M$ and
Assertion 1 of Lemma \ref{Lemma_closed} that $\vec{h}$ belongs to $V$ if and
only if the Schur functions $s_{\lambda}$ with $\lambda$ of maximal hook
length less or equal to $k$ evaluated at $\vec{h}$ are positive.\ Thus the
criterion to test the positivity of our morphisms reduces to Schur functions
and can be performed without using the $k$-Schur functions.

\item By Theorem \ref{Th_UVhomeo} we are able to compute $g=f^{-1}$ from the
Perron Frobenius vectors of the matrices $\Phi$. So our Theorem \ref{Th_RiRef}
permits in fact to compute the nonnegative Toeplitz matrix associated to any
point of $\overline{U}$ (i.e. to reconstruct $M$ from the datum of the minors
$r_{1},\ldots,r_{k}$).
\end{enumerate}
\end{remarks}

\section{Perspectives}
We expect that most of the results contained in this paper can be extended to types $B,C,D$. 
Indeed symplectic and orthogonal analogues of $k$-Schur functions have been introduced in \cite{LSS} and \cite{P}. 
They satisfy similar Pieri rules and relevant rectangle factorizations which are crucial 
ingredients in our proofs. This should permit to study central alcove random walks in the Weyl chambers of types $B,C,D$.
Another interesting problem is to consider these random walks without the restriction to stay in the Weyl chamber. 
We ignore if the graph of the weak Bruhat order (analogue of the $k$-partition poset in this setting) is then multiplicative and 
if so, for which underlying commutative algebra. Nevertheless, one can expect to reduce the problem to 
random alcove walks in Weyl chambers by purely probabilistic arguments.

\bigskip

\noindent\textbf{Acknowledgments:} Both authors thank the international
French-Mexican laboratory LAISLA, the CIMAT and the IDP for their support and
hospitality. We are also grateful to J. Guilhot for numerous discussions on
the combinatorics of alcoves and T. Lam for indicating us references which
could permit to extend results of the paper beyond type $A$.

\bigskip

\noindent C\'{e}dric Lecouvey: Institut Denis Poisson (UMR CNRS 7013).

\noindent Universit\'{e} de Tours Parc de Grandmont, 37200 Tours, France.

\noindent{cedric.lecouvey@lmpt.univ-tours.fr}

\bigskip

\noindent Pierre Tarrago: CIMAT.

\noindent36023 Guanajuato, Mexico.

\noindent{pierre.tarrago@cimat.mx}


\begin{thebibliography}{99}                                                                                               %


\bibitem {BBTZ}\textsc{C.\ Berg, N.\ Bergeron, H.\ Thomas, H and
M.\ Zabrocki,}\ Expansion of $k$-Schur functions for maximal rectangles within
the affine nilCoxeter algebra, Journal of Combinatorics, 3(3), 563-589 (2012).

\bibitem {BFZ}\textsc{A.\ Berenstein, S.\ Fomin and A. zelevinsky,}
Parametrization of canonical bases and totally positive matrices, Adv. in
Math, 122(1), 49-149 (1996).

\bibitem {FZ}\textsc{S.\ Fomin and A. zelevinsky,} Total positivity: test and
parametrizations, The Mathematical Intelligencer 22, 23-33 (2000).

\bibitem {Ker}\textsc{S. V.\ Kerov}, \textit{Asymptotic representation theory
of the symmetric group and its applications in analysis}, AMS Translation of
Mathematical Monographs, vol 219 (2003).

\bibitem {KR}\textsc{M.\ Kreuzer and L. Robbiano, }\textit{Computational
Commutative Algebra 2, }in Computational Commutative Algebra vol 1, Springer (2005).

\bibitem {LamLR}\textsc{T. Lam}, Schubert polynomials for the affine
Grassmannian, J.\ Amer. Math. Soc., 21(1), 259-281 (2008).

\bibitem {Lam}\textsc{T. Lam}, Affine Stanley symmetric functions, American
Journal of Mathematics, 128(6), 1553-1586 (2006).

\bibitem {Lam2}\textsc{T. Lam}, The shape of a random affine Weyl group
element and random core partitions, Annals of Probability, 43(4), 1643-1662 (2015).

\bibitem {LSS}\textsc{T. Lam, A. Schilling and M. Shimozono, }Schubert
polynomials for the affine Grassmannian of the symplectic group, Math.
Zeit.\ 264 (4), 765-811 (2010).

\bibitem {LLMSSZ}\textsc{T. Lam, L. Lapointe, J. Morse, A. Schilling, M.
Shimozono and M. Zabrocki}, \textit{k-Schur functions and affine Schubert
calculus}, Fields Institute Monographs, Springer (2014).

\bibitem {LLM}\textsc{L. Lapointe, A. Lascoux and J. Morse}, Tableaux atoms
and a new Macdonald positivity conjecture, Duke Math., 116(1), 103-146 (2003).

\bibitem {LLP2}\textsc{C. Lecouvey, E. Lesigne and M. Peign\'{e}}, Conditioned
random walks from Kac-Moody root systems, Transactions of the AMS., 368(5),
3177-3210 (2016).

\bibitem {LT}\textsc{C. Lecouvey and P. Tarrago}, Harmonic functions on
multiplicative graphs and weight polytopes of representations, arXiv 1609.00138.

\bibitem {Mac}\textsc{I. G.\ Macdonald}, \textit{Symmetric functions}, Oxford
Science Publications (1998).

\bibitem{Mitt} \textsc{J. Mittmann}, Independence in algebraic complexity theory, Doctoral dissertation, Universitäts-und Landesbibliothek Bonn (2013).

\bibitem {MS}\textsc{J. Morse and A. Schilling}, Crystal approach to affine
Schubert Calculus, IMRN, 8, 2239-2294 (2016).

\bibitem {OC1}\textsc{N. O' Connell}, A path-transformation for random walks
and the Robinson-Schensted correspondence, Trans. Amer. Math. Soc., 355,
3669-3697 (2003).

\bibitem {P}\textsc{N. Pon, }Affine Stanley symmetric functions for classical
types, Journal of Algebraic Combinatorics, 36(4), 595-622 (2012).

\bibitem{Pan} \textsc{A. Pandey},  Algebraic Independence: Criteria and Structural Results over Diverse Fields, Doctoral dissertation, Indian Institute of Technology Kanpur (2015).

\bibitem {Ri}\textsc{K.\ Rietsch}, Totally positive Toeplitz matrices and
quantum cohomology of partial flag varieties, JAMS, 16(2), 363-392 (2002).

\bibitem {Ri2}\textsc{K.\ Rietsch}, A mirror construction for the totally nonnegative part of the Peterson variety, Nagoya Mathematical Journal, 183, 105-142 (2006).
\end{thebibliography}
\end{document}